\documentclass[11pt, a4paper]{amsart}
\usepackage[utf8]{inputenc}
\usepackage{amsmath}
\usepackage{amssymb}
\usepackage{amsfonts}
\usepackage{amscd}
\usepackage{enumerate}
\usepackage[margin = 0.9in]{geometry}
\usepackage{amsthm}
\usepackage{mathrsfs}
\usepackage{mathtools}
\usepackage{amssymb}
\usepackage[all]{xy}
\usepackage{xcolor}
\usepackage{graphics}
\usepackage{lscape}
\usepackage{array}
\usepackage{microtype}
\usepackage{setspace}
\usepackage{adjustbox}
\usepackage{palatino}
\usepackage{eulervm}
\usepackage{parskip}
\usepackage{marvosym}
\usepackage{tikz-cd} 
\usetikzlibrary{graphs,decorations.pathmorphing,decorations.markings}
\usepackage{stmaryrd} 
\usepackage{upgreek} 
\usepackage{centernot} 
\usepackage[utf8]{inputenc}
\usepackage{comment}
\usepackage[shortlabels]{enumitem}
\usepackage{xy}

\makeatletter
\def\l@subsection{\@tocline{2}{0pt}{2.5pc}{5pc}{}}
\makeatother

\numberwithin{equation}{section}
\newtheorem*{theorem*}{Theorem}
\newtheorem*{definition*}{Definition}

\newtheorem{theorem}{Theorem}[section]
\newtheorem{lemma}[theorem]{Lemma}
\newtheorem{proposition}[theorem]{Proposition}
\newtheorem{corollary}[theorem]{Corollary}

\newtheorem{remark}[theorem]{Remark}
\theoremstyle{definition}
\newtheorem{definition}[theorem]{Definition}
\newtheorem{def/prop}[theorem]{Definition/Proposition}

\theoremstyle{remark}

\usepackage{hyperref}\hypersetup{colorlinks}

\usepackage{color} 

\definecolor{darkred}{rgb}{1,0,0} 
\definecolor{darkgreen}{rgb}{0,1,0}
\definecolor{darkblue}{rgb}{0, 0, 1}
\definecolor{darkpurple}{RGB}{170, 51, 106}

\hypersetup{colorlinks,
linkcolor=darkblue,
filecolor=darkgreen,
urlcolor=darkred,
citecolor=darkpurple}

\DeclareMathOperator{\Hom}{Hom}

\DeclareMathOperator{\End}{End}

\DeclareMathOperator{\id}{1}
\DeclareMathOperator{\Spec}{Spec}
\DeclareMathOperator{\Pic}{Pic}

\DeclareMathOperator{\Coh}{Coh}

\DeclareMathOperator{\Higgs}{Higgs}
\DeclareMathOperator{\Bun}{Bun}

\DeclareMathOperator{\Maps}{Maps}
\DeclareMathOperator{\anMaps}{anMaps}

\DeclareMathOperator{\ad}{ad}

\DeclareMathOperator{\Loc}{Loc}
\DeclareMathOperator{\coker}{coker}

\DeclareMathOperator{\Sym}{Sym}
\DeclareMathOperator{\ev}{ev}

\DeclareMathOperator{\Rep}{Rep}
\DeclareMathOperator{\QCoh}{QCoh}
\DeclareMathOperator{\Mod}{Mod}

\DeclareMathOperator{\QMod}{QMod}
\DeclareMathOperator{\Hgs}{Hgs}

\DeclareMathOperator{\Schemes}{\mathbf{Sch}}

\DeclareMathOperator{\St}{\mathbf{St}}
\DeclareMathOperator{\dSt}{\mathbf{dSt}}

\DeclareMathOperator{\Vect}{Vect}
\DeclareMathOperator{\Hodge}{Hodge}
\DeclareMathOperator{\Deligne}{Deligne}
\DeclareMathOperator{\Tw}{Tw}

\DeclareMathOperator{\Triv}{Triv}

\DeclareMathOperator{\rank}{rank}

\DeclareMathOperator{\Aut}{Aut}

\DeclareMathOperator{\NAH}{NAH}

\DeclareMathOperator{\Perf}{Perf}
\newcommand{\heart}{\ensuremath\heartsuit}
\DeclareMathOperator{\Dol}{Dol}
\DeclareMathOperator{\dR}{dR}
\DeclareMathOperator{\Hod}{Hod}
\DeclareMathOperator{\Sim}{Sim}
\DeclareMathOperator{\Del}{Del}
\DeclareMathOperator{\Bon}{Bon}
\DeclareMathOperator{\QC}{QC}
\DeclareMathOperator{\DCoh}{DCoh}
\DeclareMathOperator{\Hit}{Hit}

\DeclareMathOperator{\BBB}{BBB}
\DeclareMathOperator{\sst}{sst}
\DeclareMathOperator{\st}{st}

\DeclareMathOperator{\DMod}{D-Mod}
\DeclareMathOperator{\im}{im}

\DeclareMathOperator{\DHgs}{DHgs}
\DeclareMathOperator{\PHgs}{PHgs}

\DeclareMathOperator{\nilp}{nilp}
\DeclareMathOperator{\NHgs}{NHgs}
\DeclareMathOperator{\Ker}{Ker}
\DeclareMathOperator{\Sect}{Sect}

\newcommand{\Bb}{\mathcal{B}}

\newcommand{\Dd}{\mathcal{D}}
\newcommand{\Ee}{\mathcal{E}}
\newcommand{\Ff}{\mathcal{F}}
\newcommand{\Gg}{\mathcal{G}}

\newcommand{\Jj}{\mathcal{J}}

\newcommand{\Mm}{\mathcal{M}}

\newcommand{\Oo}{\mathcal{O}}

\newcommand{\Uu}{\mathcal{U}}
\newcommand{\Vv}{\mathcal{V}}

\newcommand{\Zz}{\mathcal{Z}}

\newcommand{\D}{\mathbf{D}}

\newcommand{\B}{\mathrm{B}}

\newcommand{\f}{\mathrm{f}}

\renewcommand{\d}{\mathrm{d}}

\newcommand{\ol}[1]{\overline{#1}}

\newcommand{\wt}[1]{\widetilde{#1}}

\renewcommand{\AA}{\mathbb{A}}
\newcommand{\BB}{\mathbb{B}}
\newcommand{\CC}{\mathbb{C}}
\newcommand{\DD}{\mathbb{D}}

\newcommand{\GG}{\mathbb{G}}

\newcommand{\PP}{\mathbb{P}}

\newcommand{\TT}{\mathbb{T}}

\newcommand{\gG}{\mathfrak{g}}

\renewcommand{\to}{\longrightarrow}

\newcommand\Quotient[2]{
\mathchoice
{
\text{\raise1ex\hbox{\thinspace $#1$}\Big{/} \lower1ex\hbox{$#2$} \thinspace}%
}
{
#1\,/\,#2
}
{
#1\,/\,#2
}
{
#1\,/\,#2
}
}

\newcommand\GIT[2]{
\mathchoice
{
\text{\raise1ex\hbox{\thinspace $#1$}\Big{/}\!\!\!\!\Big{/} \lower1ex\hbox{$#2$} \thinspace}%
}
{
#1\,/\,#2
}
{
#1\,/\,#2
}
{
#1\,/\,#2
a       }
}

\thanks{
First author partially supported by the Spanish Ministry of Science and Innovation, through the ‘Severo Ochoa Programme for Centres of Excellence in R$\&$D’ (CEX2019-000904-S), and through project PGC2022-1001150218, and by the Portuguese FCT through CEECIND/04153/2017. Second author supported by La Caixa INPhINIT programme with fellowship number 1801P.01033.
}

\begin{document}

\title[The Dirac-Higgs complex and categorification of (BBB)-branes]{\bf The Dirac-Higgs complex and categorification of (BBB)-branes}

\author[E. Franco]{Emilio Franco}
\address{E. Franco,
\newline\indent Depto. Matem\'aticas, Facultad de Ciencias, 
\newline\indent Universidad Aut\'onoma de Madrid
\newline\indent Campus de Cantoblanco 28049, Madrid, Espa\~na.}
\email{emilio.franco@uam.es}

\author[R. Hanson]{Robert Hanson}
\address{R. Hanson, \newline\indent Centro de An\'alise Matem\'atica, Geometria e Sistemas Din\^{a}micos, 
\newline\indent Instituto Superior T\'ecnico, Universidade de Lisboa, 
\newline\indent Av. Rovisco Pais s/n, 1049-001 Lisboa, Portugal}
\email{robert.hanson@tecnico.ulisboa.pt}

\begin{abstract}
Let $\Mm_{\Dol}(X,G)$ denote the hyperk\"ahler moduli space of $G$-Higgs bundles over a smooth projective curve $X$. In the context of four dimensional supersymmetric Yang-Mills theory, Kapustin and Witten introduced the notion of (BBB)-brane: boundary conditions that are compatible with the B-model twist in every complex structure of $\Mm_{\Dol}(X,G)$. The geometry of such branes was initially proposed to be  hyperk\"ahler submanifolds that support a hyperholomorphic bundle. Gaiotto has suggested a more general type of (BBB)-brane defined by perfect analytic complexes on the Deligne-Hitchin twistor space $\Tw(\Mm_{\Dol}(X,G))$. 

Following Gaiotto’s suggestion, this paper proposes a framework for the categorification of (BBB)-branes, both on the moduli spaces and on the corresponding derived moduli stacks. We do so by introducing the \textit{Deligne stack}, a derived analytic stack with corresponding moduli space $\Tw(\Mm_{\Dol}(X,G))$, defined as a gluing between two analytic Hodge stacks along the Riemann-Hilbert correspondence. We then construct a class of (BBB)-branes using integral functors that arise from higher non-abelian Hodge theory, before discussing their relation to the Wilson functors from the Dolbeault geometric Langlands correspondence. 
\end{abstract}

\maketitle

{
  \hypersetup{linkcolor=black}
  \setcounter{tocdepth}{2}
  \tableofcontents
}

\newpage

\section{Introduction}

\subsection{Background}

Hitchin's classical papers from 1987 \cite{hitchin_self, hitchin_integrable} derive a dimensional reduction of the self-duality equations and describe a space of solutions with a rich geometric structure. By introducing an auxillary \textit{Higgs field}, the parameter space $\Higgs_G(X)$ consists of the moduli stack of $G$-Higgs bundles on a smooth projective curve $X$, whilst the space of solutions $\Mm_{\Dol}(X,G)$ is the moduli space of Geiseker-semistable objects. On the moduli space, the remarkable work of Hitchin includes the construction of a hyperk\"ahler metric and an algebraic integrable system, since becoming central objects of study in many areas of mathematics and physics.

Courtesy of the hyperk\"ahler metric, $\Mm_{\Dol}(X,G)$ plays a special role in four dimensional supersymmetric Yang-Mills theories. After a dimensional reduction from four to two, S-duality acts as mirror symmetry between $\Mm_{\Dol}(X,G)$ and $\Mm_{\Dol}(X, \, ^LG)$, exchanging $G$ for the Langlands dual group $^LG$, whilst inverting a coupling constant and exchanging electric and magnetic charges \cite{vafa, GNO, strominger, MO}. This is the starting point for the revolutionary work of Kapustin-Witten \cite{kapustin&witten} that relates S-duality in intricate ways to the geometric Langlands corrspondence of Beilinson-Drinfeld \cite{beilinson&drinfeld}. Initially described as a "best hope"\footnote{A later refinement by Arinkin-Gaitsgory \cite{arinkin&gaitsgory} has recently been proven in a remarkable series of papers \cite{ABC+}.}, Beilinson-Drinfeld proposed an equivalence 
\begin{equation}
\label{eq dR GLC intro}
\DMod(\Bun_G(X)) \xrightarrow{\cong} \DCoh(\Loc_{^LG}(X)),
\end{equation}
between $D$-modules on $\Bun_G(X)$ and $\Oo$-modules on $\Loc_{^LG}(X)$, on the moduli stacks of $G$-bundles and flat connections, respectively. In Kapustin and Witten's theory, $D$-modules arise from Lagrangian branes via the quantisation of the \textit{canonical coisotropic brane} \cite[§11.2]{kapustin&witten}. In this way, they sought to reinterpret \eqref{eq dR GLC intro} as a duality between boundary conditions related by S-duality. 

These boundary conditions undergo hyperk\"ahler enhancements of the usual B-brane and A-branes found in homological mirror symmetry \cite{kontsevich}. Kapustin-Witten introduce \textit{(BBB)-branes:} hyperk\"ahler submanifolds that support a hyperholomorphic bundle, while their mirrors are proposed to be \textit{(BAA)-branes}: holomorphic Lagrangian submanifolds that support a spin bundle. The interchange of (BBB) and (BAA)-branes can be understood via Fourier-Mukai transforms that arise in relation to Donagi-Pantev's \textit{classical limit} of \eqref{eq dR GLC intro}, which proposes a derived equivalence 
\begin{equation}
\label{eq Dol GLC intro}
\DCoh(\Higgs_G(X)) \to \DCoh(\Higgs_{^LG}(X)).
\end{equation}
With Fourier-Mukai transforms that resemble the classical limit, various authors have addressed the interchange of (BBB) and (BAA)-branes on $\Mm_{\Dol}(X,G)$ and $\Mm_{\Dol}(X, \, ^LG)$. Several examples can be found within the regular locus \cite{baraglia&schaposnik1, baraglia&schaposnik2, BCFG, biswas&garcia, HMP, heller&schaposnik, hitchin_char}, whilst the duality is significantly more technical in the discriminant locus \cite{FGOP, franco&peon}. 

A complete understanding of the role of physics in geometric Langlands theory requires a description of the boundary conditions in terms of the derived geometry of the stacks. This is part of the philosophy behind the Betti geometric Langlands of Ben-Zvi--Nadler \cite{ben-zvi&nadler}, in which a category of microsheaves plays the role of derived Lagrangian branes. An alternative perspective in the language of shifted symplectic structures \cite{PTVV} is hinted at by Ginzburg-Rozenblyum \cite{ginzburg&rozenblyum}, who construct a family of derived Lagrangian substacks that generalise (BAA)-branes constructed by Gaiotto on the moduli spaces \cite{gaiotto}. It is then natural to ask for a mirror/Langlands dual notion of derived (BBB)-brane on $\Higgs_G(X)$. 

\subsection{Summary of the paper}

Our starting point is a suggestion of Gaiotto on the categorification of (BBB)-branes \cite[Appendix C]{gaiotto}. This involves both algebraic and analytic topologies, so we shall work over both the smooth projective curve $X$ and the underlying analytic curve $\Sigma=X^{an}$. 

The main idea is to apply the twistor transform of Kaledin-Verbitsky \cite{kaledin&verbitsky} to consider (BBB)-branes on $\Mm_{\Dol}(\Sigma,G)$ as sheaves on the twistor space $\Tw(\Mm_{\Dol}(\Sigma,G))$ that are trivial on horizontal twistor lines. In this way, the study of (BBB)-branes can be embedded within twistor theory. Gaiotto suggests a wider class of (BBB)-brane defined by perfect complexes on $\Tw(\Mm_{\Dol}(\Sigma,G))$. The \textit{perfect} condition arises from the behaviour of certain supersymmetric ground states and their excitations. We therefore consider 
\[
\Perf^{\, =}(\Tw(\Mm_{\Dol}(\Sigma, G))) \subset \Perf(\Tw(\Mm_{\Dol}(\Sigma, G))),
\]
the full subcategory of complexes $\TT \in \Perf(\Tw(\Mm_{\Dol}(\Sigma, G)))$ that satisfy the following property: for every horizontal twistor line $\sigma : \PP^1 \to \Tw(\Mm_{\Dol}(\Sigma, G))$, the pullback $\sigma^{*}\TT \in \Perf(\PP^1)$ is quasi-isomorphic to a complex of free sheaves. This is simply a condition on the induced transition functions over the equator of the Riemann sphere (or rather the "algebraic equator" $\CC^{*} \subset \PP^1$).

Alongside an analytic complex $\TT$ on the twistor space, we also consider (BBB)-brane data to consist of an underlying algebraic complex $\BB \in \Perf(\Mm_{\Dol}(X, G)))$, related to $\TT$ via restriction and analytification. Our proposal for a category of (BBB)-branes is therefore defined by the Cartesian diagram 
\begin{equation}
\label{eq bbb category intro2}
\begin{tikzcd}
\BBB( \Mm_{\Dol}(X, G) )  
\arrow[dr, phantom, "\square"] 
\arrow[r]
\arrow[d]
& \Perf^{\, =}(\Tw(\Mm_{\Dol}(\Sigma, G))) \arrow[d, "L\imath_{\Dol}^{*}"] \\
\Perf(  \Mm_{\Dol}(X , G) ) \arrow[r, "(\cdot)^{an}"'] & \Perf( \Mm_{\Dol}(\Sigma, G)  )
\end{tikzcd}.
\end{equation}
A (BBB)-brane is a then a pair $(\BB, \TT)$ such that $\TT$ encodes the \textit{hyperk\"ahler rotations} of $\BB$. By the \textit{categorification of (BBB)-branes}, we refer to:

\textbf{Proposition A.} (\textit{Proposition \ref{co BBB category}). The (BBB)-branes on $\Mm_{\Dol}(X,G)$ defined by Kapustin-Witten correspond to locally free objects in $\BBB(\Mm_{\Dol}(X,G))$ that are supported in one homological degree.}

The rich geometry of the twistor space arises from the non-abelian Hodge theory of Corlette, Donaldson, Hitchin and Simpson \cite{corlette, donaldson, hitchin_self, simpson_higgs}. In letters between Deligne and Simpson which led to Simpson's paper \cite{simpson_hodge_2}, the construction of $\Tw(\Mm_{\Dol}(\Sigma, G))$ was elegantly expressed in terms of \textit{$\lambda$-connections}: a linear interpolation between flat connections ($\lambda = 1$) and Higgs fields ($\lambda=0$). 

To recall the construction of $\Tw(\Mm_{\Dol}(\Sigma,G))$, let $\Mm_{\Hod}(\Sigma , G)$, $\Mm_{\dR}(\Sigma , G)$ and $\Mm_{B}(\Sigma , G)$ denote the analytic moduli spaces of $\lambda$-connections, flat connections and representations $\pi_1(\Sigma) \to G$, respectively. The Riemann-Hilbert correspondence assigns the holonomy representation to a connection, giving rise to an isomorphism $RH : \Mm_{B}(\Sigma , G) \cong \Mm_{\dR}(\Sigma , G).$ A natural $\GG_m$-action rescales $\lambda$-connections, resulting in an embedding 
\[
\Mm_{B}(\Sigma , G) \times (\AA^1 - 0) \xrightarrow{RH \times \id} \Mm_{\dR}(\Sigma , G) \times (\AA^1 - 0) \hookrightarrow \Mm_{\Hod}(\Sigma , G)
\]
We apply the same constructions over $\ol{\Sigma}$, the analytic curve with complex conjugate holomorphic structure. Note that $\Sigma$ and $\ol{\Sigma}$ share the same underlying topological space, so $\Mm_{B}(\Sigma , G) = \Mm_{B}(\ol{\Sigma} , G)$. This allows us to view the Betti moduli spaces as a subspace of the Hodge moduli spaces over both $\Sigma$ and $\ol{\Sigma}$. The \textit{Deligne-Hitchin twistor space} is then defined as a gluing  
\begin{equation}
\label{eq define glu moduli intro}
\begin{tikzcd}
\Mm_{B}(\Sigma , G) \times (\AA^1 - 0) = \Mm_{B}(\ol{\Sigma} , G) \times (\AA^1 - 0)
\arrow[r, hook]
\arrow[d, hook]
\arrow[dr, phantom, "\square"]
& \Mm_{\Hod}(\ol{\Sigma},G) \arrow[d] \\
\Mm_{\Hod}(\Sigma,G) \arrow[r] & \Tw(\Mm_{\Dol}(\Sigma,G))
\end{tikzcd}.
\end{equation}
With this construction at hand, we can state the central philosophy of this paper: to explore the extension the (co)-fiber products \eqref{eq define glu moduli intro} and \eqref{eq bbb category intro2} to the realm of 
of derived analytic and algebraic geometry. The stacks at play arise from Simpson's \textit{higher non-abelian Hodge theory} \cite{simpson_hodge_2, simpson_dolbeault, simpson_dolbeault_2,simpson_hodge}, which studies a family of mapping stacks that encode higher cohomology theories:  
\begin{center}
\begin{tabular}{c|c|c}
Simpson shapes $X_{\Sim}$ as &
Derived moduli stacks of & 
Moduli spaces of  
\\
classical 1-stacks &
the form $\Maps(X_{\Sim} , BG)$ 
& semistable objects \\
\hline
$X_{\Dol}$ & $\Higgs_G(X)$ & $\Mm_{\Dol}(X,G)$ \\
$X_{\dR}$  & $\Loc_G(X)$ & $\Mm_{\dR}(X,G)$ \\
$X_{\Hod}$ & $\Hodge_G(X)$ & $\Mm_{\Hod}(X,G)$ \\
$X_{B}$    & $\Rep_G(X)$ & $\Mm_{B}(X,G)$ \\

\end{tabular} 
\end{center}
These stacks can be constructed either in algebraic or analytic topologies, where the latter is considered over $\Sigma=X^{an}$. The theory of derived analytificaton developed by Porta-Yu \cite{porta_GAGA, porta&yu} allows us to pass between these two topologies. In the derived Riemann-Hilbert correspondence of Holstein-Porta \cite{holstein&porta, porta}, they construct a remarkable morphism $\nu_{RH} : \Sigma_{\dR} \to \Sigma_{B}$ and study the equivalence 
\[
RH = \nu_{RH}^{*}: \Rep_G(\Sigma) \to \Loc_G(\Sigma),
\]
as a derived extension of the holonomy representation map $RH : \Mm_B(\Sigma, G) \xrightarrow{\cong} \Mm_{\dR}(\Sigma, G)$. This provides us with a natural way of constructing a derived extension of the Deligne and Simpson's construction of the twistor space. 

\textbf{Definition B.} \textit{(Definition \ref{de deligne stack}). Let $\Sigma$ be a complex analytic curve and $G$ be a reductive group. The corresponding \textit{Deligne stack} $\Deligne_G(\Sigma)$ is the complex analytic stack defined by the pushout} 
\begin{equation*}
\begin{tikzcd}[column sep = huge]
\Rep_G(\Sigma) \times (\AA^1 - 0) \arrow[dr, phantom, "\square"] \arrow[d, hook] \arrow[r, hook]  & \Hodge_G(\ol{\Sigma})  \arrow[d, hook] \\
\Hodge_G(\Sigma) \arrow[r, hook]  & \Deligne_G(\Sigma)
\end{tikzcd}.
\end{equation*}

With the Deligne stack at hand, we replicate the category of (BBB)-brane defined in \eqref{eq bbb category intro2} to propose the following notion of derived (BBB)-brane. 

\textbf{Definition C.} \textit{(Definition \ref{de stacky BBB})}.
\textit{A (BBB)-brane on $\Higgs_G(X)$ is an object of the dg-category defined by the fiber product} 
\begin{equation}
\label{eq bbb category intro}
\begin{tikzcd}
\BBB( \Higgs_G(X) )  
\arrow[dr, phantom, "\square"] 
\arrow[r]
\arrow[d]
& \Perf^{\, =}(\Deligne_G(\Sigma)) \arrow[d, "L\imath_{\Dol}^{*}"] \\
\Perf(  \Higgs_G(X))  \arrow[r, "(\cdot)^{an}"'] & \Perf( \Higgs_G(\Sigma)  )
\end{tikzcd}.
\end{equation}
We then turn to the construction of examples. Our main example is inspired by a historically significant example of a (BBB)-brane - the \textit{Dirac-Higgs bundle}. This is a hyperholomorphic vector bundle constructed by Hitchin \cite{hitchin_dirac} by considering a null space bundle for a family of Dirac operators 
\[
\D_{(E,\phi)}^* = \begin{pmatrix} \overline{\partial}_E &  \phi \\ \phi^* & \partial_E \end{pmatrix} .
\]
Inspired by an algebraic interpretation of Hausel \cite{hausel}, we define the \textit{Dirac-Higgs complex} in terms of the universal $G$-Higgs bundle $\Uu \to X \times \Higgs_G(X)$, an auxillary representation $\rho : G \to GL_n$ and pushforward along the projection map $p_2 : X \times \Higgs_G(X) \to \Higgs_G(X)$,
\begin{equation}
\label{eq dirac-higgs intro}
\DD_{\rho} := Rp_{2,*}(\rho(\Uu)) \in \Perf(\Higgs_G(X)).
\end{equation}
By adding a twist, we also consider a family of complexes parameterised by a Higgs bundle on $X$, $(E,\phi) \in \Hgs(X,K_X)$. We consider complexes
\begin{equation}
\label{eq dol branes intro2}
\Phi_{\Bon, \rho}(E,\phi) := Rp_{2,*}(\rho(\Uu) \otimes (E,\phi)) \in \Perf(\Higgs_G(X)).
\end{equation}
which resemble hyperholomorphic bundles on the moduli spaces studied by Bonsdorff, Franco-Jardim and Frejlich-Jardim \cite{bonsdorff_1, bonsdorff_2, franco&jardim, frejlich&jardim}. 

To show that \eqref{eq dirac-higgs intro} and \eqref{eq dol branes intro2} are in accordance with our notion of (BBB)-brane, we require a method of constructing complexes on the Deligne stack that satisfy the described functorial compatibilities. Natural candidates for such constructions are integral functors with universal families as kernels, for all the mapping stacks $\Maps(\Sigma_{\Sim}, BG)$, where $\Sim \in \{ \Dol, \dR, \Hod, B \}$ is a formal variable\footnote{The notation "$\Sim$" for this variable is chosen in homage to Carlos Simpson}. In this sense, our branes are constructed by integrating over higher non-abelian Hodge theory. 

The following diagrams are a sketch of our constructions. We take universal families $\Vv_{\Sim} \to \Sigma_{\Sim} \times \Maps(\Sigma_{\Sim}, BG)$ as integral kernels to define the \textit{Dolbeault, de Rham, Hodge and Betti functors.} In the first diagram, the diagonal arrows are the Betti and Hodge functors, whilst the front and back squares are Cartesian. We fill out the first dashed arrow to define the \textit{Deligne functor} as a gluing of complex conjugate Hodge functors. We fill out the second dashed arrow to construct (BBB)-branes. 
\begin{equation*}
\adjustbox{scale=0.92,center}{%
\begin{tikzcd}[row sep={40,between origins}, column sep={70,between origins}]
& 
\Perf(\Sigma_{\Hod}) \times_B \Perf(\ol{\Sigma}_{\Hod}) 
\ar{rr}\ar{dd} 
\arrow[dl, dashed, "\Psi_{\Del, \rho}"'] 
& & 
\Perf(\ol{\Sigma}_{\Hod})
\vphantom{\times_{S_1}} \ar{dd} \ar[dl, "\Psi_{\ol{\Hod}, \rho}"]
\\
\Perf(\Deligne_G(\Sigma)) \ar[crossing over]{rr} \ar{dd} & & \Perf(\Hodge_G(\ol{\Sigma}))
\\
& \Perf(\Sigma_{\Hod}) \ar{rr} \ar[dl, "\Psi_{\Hod, \rho}"] & &  \Perf(\Sigma_{B} \times (\AA^1 - 0)) \vphantom{\times_{S_1}} \ar[dl, "\widehat{\Psi}_{B,\rho}"] 
\\
\Perf(\Hodge_G(\Sigma)) \ar{rr} && \Perf(\Rep_G(\Sigma) \times (\AA^1 - 0)) \ar[from=uu,crossing over]
\end{tikzcd}}
\end{equation*}
\begin{equation*}
\adjustbox{scale=0.92,center}{%
\begin{tikzcd}
\PHgs(\Sigma, K_{\Sigma}) \times_{\PHgs(X, K_X)} (\Perf(\Sigma_{\Hod}) \times_{B} \Perf(\ol{\Sigma}_{\Hod}))^{\, =} 
\arrow[ddr, bend right = 20, "\Phi_{\Bon, \rho}"'] 
\arrow[dr, dashed] 
\arrow[drr, bend left=20, "\Psi_{\Del, \rho}"] &  & \\
& \BBB( \Higgs_G(X) )  
\arrow[dr, phantom, "\square"] 
\arrow[r]
\arrow[d]
& \Perf^{\, =}(\Deligne_G(\Sigma)) \arrow[d, "L\imath_{\Dol}^{*}"] 
\\
& \Perf(  \Higgs_G(X))  \arrow[r, "(\cdot)^{an}"'] & \Perf( \Higgs_G(\Sigma)  )
\end{tikzcd}}
\end{equation*}

Unwrapping this notation, the initial data for our (BBB)-branes consists of a vector
\[
(E,\phi, \Ff, \ol\Ff, f),
\] 
such that 
\begin{itemize}
    \item $(E,\phi) \in \PHgs(X, K_X)$ is an algebraic complex of locally free Higgs sheaves on $X$.
    
    \item $(\Ff,\ol\Ff, f) \in \Perf(\Sigma_{\Hod}) \times_B \Perf(\ol\Sigma_{\Hod})$ is a product of perfect analytic complexes on Hodge shapes (see \eqref{eq Dom}).

    \item Moreover $(\Ff,\ol\Ff, f) \in \left( \Perf(\Sigma_{\Hod}) \times_B \Perf(\ol\Sigma_{\Hod}) \right)^{\,=}$ applies a certain condition involving horizontal twistor lines (see \eqref{eq hodge =}). 
    
    \item $\Ff$ restricts to the analytification $(E,\phi)^{an}$. 
\end{itemize}

\textbf{Theorem D.} \textit{(Theorem \ref{th BBB}, Corollary \ref{co Dirac-Higgs is BBB}). Given the data just described, the pair
\[
\BB := \Phi_{\Bon,\rho}(E,\phi) \in \Perf(\Higgs_G(X)),
\]
\[
\TT := \Psi_{\Del, \rho}(\Ff, \ol\Ff, f) \in \Perf^{\, =}(\Deligne_G(\Sigma)),
\]
defines an object $(\BB, \TT) \in \BBB(\Higgs_G(X))$. We call such objects \textbf{integral branes.} Moreover, this family includes the Dirac-Higgs complex as the special case $(E,\phi) = (\Oo_X, 0)$ and $(\Ff, \ol\Ff, f) = (\Oo_{\Sigma_{\Hod}}, \Oo_{\ol{\Sigma}_{\Hod}}, \id)$.}

In our final result we study the relation between our integral functor constructions and the Wilson functors from Donagi-Pantev's \textit{classical limit} of geometric Langlands. Given a $\CC$-valued point $x \in X$ and a representation $\mu : G \to GL_n$, they define Wilson functors $W_{\mu, x}$ as a degeneration of Wilson functors that appear in Beilinson-Drinfeld's original work. We place $W_{\mu, x}$ in the following commutative diagram.

\textbf{Theorem E.} (\textit{Theorem \ref{th exchange wilson hecke}}). 
Fix points $x \in X$ and $y = \iota_{\Dol}(x) \in X_{\Dol}$ related by the inclusion $\iota_{\Dol} : X \to X_{\Dol}$. Let $\rho, \mu : G \to GL_n$ be two representations. The Dolbeault integral functors with respect to $\rho$ and $\rho \otimes \mu$ fit into the following commutative square: 
\begin{equation*}
\begin{tikzcd}[column sep = huge]
\DCoh(X_{\Dol}) \arrow[r, "\Phi_{\Dol, \rho}"] \arrow[d, "(\cdot)|_{y}"'] & \DCoh(\Higgs_G(X)) \arrow[d, "W_{\mu, x}"] \\
\DCoh(X_{\Dol}) \arrow[r, "\Phi_{\Dol, \rho \otimes \mu}"'] & \DCoh(\Higgs_G(X))
\end{tikzcd}.
\end{equation*}

\subsection{Structure}

Section \ref{se hodge} is a short review of stacks that appear in non-abelian Hodge theory. These are  derived mapping stacks of the form $\Maps(\, \cdot \,, BG)$, which parameterise Higgs bundles, flat connections, $\lambda$-connections and holonomy representations. Section \ref{se Deligne stack} passes to the corresponding moduli spaces and recalls the construction of the Deligne-Hitchin twistor space. Using the stacks defined in Section \ref{se hodge}, we then generalise the twistor constructions to the setting of derived analytic stacks, which provides the definition of the \textit{Deligne stack}. 

Section \ref{se dirac-higgs} begins by recalling Hitchin's analytic construction of the Dirac-Higgs bundle on $\Mm_{\Dol}(X,G)$ and the subsequent algebraic intepretation of Hausel. We show how Hausel's interpretation defines a complex on $\Higgs_G(X)$ that we name the \textit{Dirac-Higgs complex}. Section \ref{se integral functors} adds twists to the Dirac-Higgs complex, interpreted as the construction and study of integral functors. Namely, we define the Bonsdorff, Dolbeault, de Rham, Hodge, Betti and Deligne functors. 

Section \ref{se BBB on Higgs} studies the suggestion of Gaiotto on the categorification of (BBB)-branes. By using the Deligne stack from Section \ref{se Deligne stack}, we provide a natural generalisation of Gaiotto's suggestion to the setting of derived algebraic and analytic stacks. Section \ref{se nahm branes} constructs examples of (BBB)-branes, with respect to the categorification of the previous section. Our examples arise from the integral functors constructed in Section \ref{se integral functors}. Section \ref{se relation} is dedicated to the relation between our (BBB)-branes and Donagi-Pantev's classical limit of the geometric Langlands correspondence. Specifically, we compute the relation between the Dolbeault integral functor and the Wilson functors in the classical limit.

\subsection{Notation and conventions}
We work over $\CC$ throughout. By \textit{category}, we refer to a $\CC$-linear dg-category, or equivalently to a stable $\CC$-linear $\infty$-category. By  \textit{algebraic stack}, we refer to a derived algebraic stack over the étale site of derived algebraic schemes. Similarly, by \textit{analytic stack}, we refer to a derived analytic stack over the étale site of derived analytic spaces. $\dSt$ and $\St$ denote the categories of derived and classical 1-stacks, respectively. 

We write $\QCoh(\cdot)$ for the dg-enhanced category of quasi-coherent sheaves and $\Vect(\cdot) \subset \Coh(\cdot) \subset \QCoh(\cdot)$ for the subcategories of locally free and coherent objects, respectively. We write $\QC(\cdot)$ for the dg-enhanced derived category of bounded, quasi-coherent sheaves and $\Perf(\cdot) \subset \DCoh(\cdot) \subset \QC(\cdot)$ for the subcategories of perfect and coherent objects, respectively.

We make use of truncations to classical 1-stacks via the truncation functor $t_0 : \dSt \to \St$. Our most general constructions are valid only in $\dSt$, as we require various fiber products to be well-behaved and various applications of base change to be valid. Compatibilities with truncation will be explicitly identified where necessary. 

\subsection{Acknowledgements}

Particular thanks are due to Anna Sisak for recommending to us the twistor theory interpretation of a (BBB)-brane. We thank Eric Chen, Peter Gothen, Pranav Pandit, Ana Peón-Nieto, Francesco Sala and Menelaos Zikidis for many helpful conversations, as well as David Ben-Zvi, Adrien Brochier, Damien Calaque, Justin Hilburn, Jon Pridham and Jason Starr for sharing advice to the second author's questions online at Math Overflow. Thanks also to Johannes Horn, Andr\'e Oliveira and João Ruano for their work on related projects and to the anonymous referee for suggesting numerous improvements to a previous version of this work.  

\section{Higher non-abelian Hodge theory}
\label{se hodge}

This section is a short review of stacks that appear in non-abelian Hodge theory. Over a smooth projective variety $X$, we define the moduli stacks of Higgs bundles, flat connections, $\lambda$-connections and holonomy representations. Their description is provided in terms of mapping stacks $\Maps(Z, BG)$, the parameter stack of morphisms $Z \times S \to BG$ in $\dSt$. The source will be given by one the of the classical 1-stacks $Z = X_{\Dol}$, $X_{\dR}$, $X_{\Hod}$, $X_B$, the \textit{Simpson shapes}. The target is always the classifying space $BG$ of a reductive group $G$.

\subsection{Dolbeault shape and Higgs bundles}
\label{se dolbeault shape}
\label{se dolbeault}
\label{se higgs is mapping stack}
Over a smooth projective variety $X$, a \textit{Higgs bundle} $(E,\phi)$ is a vector bundle $E \to X$ together with a section $\phi \in H^0(\End(E) \otimes \Omega_X)$ such that $\phi \wedge \phi = 0$, where $\phi$ is then called a \textit{Higgs field} \cite{hitchin_self, simpson_higgs}. A \textit{Higgs sheaf} is the same data, but allows $E$ to be a coherent sheaf. Let $\Hgs(X, \Omega_X)$ denote the category of Higgs sheaves. Similarly, for a line bundle $L \to X$, one can consider $L$-Higgs fields of the form $\phi \in H^0(X, \End(E) \otimes L)$, which are objects of $\Hgs(X, L)$. A Higgs sheaf has a \textit{Dolbeault complex} given by
\begin{equation}
\label{eq dol complex}
E \xrightarrow{\phi} E \otimes \Omega_X \xrightarrow{\phi} E \otimes (\Omega_X \wedge \Omega_X) \to \cdots.
\end{equation}
The cohomology defines the Dolbeault cohomology $H^{\bullet}_{\Dol}(E,\phi)$, where the usual Dolbeault cohomology of $X$ is then given by 
\[
H^{i}_{\Dol}(\Oo_X, 0) 
= H^{i}( \Oo_X \xrightarrow{0} \Omega_X \xrightarrow{0} \Omega_X \wedge \Omega_X \to \cdots )
\cong \bigoplus_{p + q = i}H^{q}(X, \Omega_X^{p}).
\]
Introduced and studied by Simpson \cite{simpson_dolbeault, simpson_dolbeault_2}, the \textit{Dolbeault shape} $X_{\Dol}$ is defined as the relative classifying stack for the formal group scheme $\widehat{TX} \to X$, where $\widehat{TX}$ denotes the formal completion of the tangent bundle along the zero section (see \cite[§3.1.3]{gaitsgory&rozenblyum} for such completions). The Dolbeault shape arises with a natural atlas on $\kappa : X \to X_{\Dol}$.
\begin{proposition} 
\label{pr QCoh Dol}
(\cite{simpson_hodge_2}). Let $X$ denote a smooth projective variety. Then pullback along $\kappa:X \to X_{\Dol}$ induces an equivalence of categories 
\[
\kappa^{\heart} : \Coh(X_{\Dol}) \xrightarrow{\cong} \Hgs(X, \Omega_X), 
\]
with an inverse $\kappa_{\heart}$ induced by pushforward. Moreover, this equivalence can be extended to families: 
\begin{equation}
\label{eq relative BNR}
(\kappa \times \id_S)^{\heart} : \Coh(X_{\Dol} \times S) \xrightarrow{\cong} \Hgs(X \times S, q_1^{*}\Omega_X),
\end{equation}
compatible with $G$-(Higgs) bundles for reductive structure group $G$. 
\end{proposition}

\begin{proof} 
It is convenient to equate Higgs fields $\phi : E \to E \otimes \Omega_X$ with actions of the symmetric algebra $\Lambda_{\Dol} = \Sym^{\bullet}(T_X)$. This is achieved by taking the dual morphism $\phi^{\vee} : E \otimes T_X \to E$ and noting that $\Lambda_{\Dol}$ is generated by $\Sym^{0}(T_X) = \Oo_X$ and $\Sym^{1}(T_X) = T_X$. As pointed out by Simpson \cite[Lemma 6.5]{simpson2}, this observation defines an equivalence 
\begin{equation*}
\Mod_{\Lambda_{\Dol}}(X) \xrightarrow{\cong} \Hgs(X, \Omega_X).
\end{equation*}
With this in place, we shall describe the first statement as a composition 
\[
\Coh(X_{\Dol}) \cong \Mod_{\Lambda_{\Dol}}(X) \cong \Hgs(X, \Omega_X). 
\]
The equivalence is implicit in work of Simpson \cite{simpson_hodge_2} and one can also refer to work of Porta-Sala \cite[Proposition 5.1.2]{porta&sala}. The idea of the proof follows from three steps. First, one makes use of the equivalence between quasi-coherent sheaves on a classifying space and comodules for its ring of regular functions \cite[Tag 06WT]{the_stacks_project}, which amounts to $\Sym^\bullet(\widehat{T^*X})$ in this case. Secondly, $\Sym^\bullet(\widehat{T^*X}) \cong \widehat{\Sym^\bullet(T^*X)}$ as the symmetric algebra of a colimit is the colimit of the symmetric algebras \cite[Tag 00DM]{the_stacks_project}, and finally the identification of $\Lambda_{\Dol} = \Sym^\bullet(TX)$  with the dual of $\widehat{\Sym^\bullet(T^*X)}$ derived from the duality between limits and colimits. The moreover statement is established by Simpson \cite[Proposition 3.1]{simpson_dolbeault_2}. 
\end{proof}
Let $X$ be a smooth projective variety and $G$ be a reductive group. We define the \textit{Higgs stack} to be  
\[
\Higgs_G(X) = \Maps(X_{\Dol},BG),  
\]
the moduli stack of $G$-bundles on $X_{\Dol}$, or equivalently, of $G$-Higgs bundles on $X$. The equivalence between these two moduli problems has the following universal incarnation. 
\begin{corollary}
\label{co universal higgs and bun}
Let $\Uu$ denote the universal $G$-Higgs bundle on $X\times \Higgs_G(X)$ and let $\Vv_{\Dol}$ denote the universal $G$-bundle on $X_{\Dol} \times \Maps(X_{\Dol}, BG).$ Let $\rho :G \to GL_n$ be a representation. There exists isomorphisms
\[
(\kappa \times \id)^{\heart} \Vv_{\Dol} \cong \Uu, \quad (\kappa \times \id)^{\heart} \rho(\Vv_{\Dol}) \cong \rho(\Uu).
\]
\end{corollary}
Grothendieck’s intended philosophy for stacks \cite{grothendieck_stacks} is the pursuit of higher cohomology theories. The case of $\Higgs_G(X)$ can be viewed as a higher form of Dolbeault cohomology, as the following result of Simpson illustrates.   
\begin{proposition}
\label{pr dol cohomology}
\cite[Proposition 3.2]{simpson_dolbeault_2}. Let $X$ be a smooth projective variety, $\Ee$ be a sheaf on $X_{\Dol}$, and let $(E, \phi) = \kappa^{\heart} \Ee$ be the corresponding Higgs sheaf on $X$ (equivalently; $\kappa_{\heart}(E, \phi) = \Ee$). Then, there exists an isomorphism between the sheaf cohomology of $\Ee$ and the Dolbeault cohomology of $(E,\phi)$;
\begin{equation}
\label{eq dol cohom}
H^{\bullet}(X_{\Dol}, \Ee) \cong H^{\bullet}_{\Dol}(E,\phi).
\end{equation}
In particular, the Dolbeault cohomology of $X$ is given by $R\Gamma( \Oo_{X_{\Dol}} ) = H^{\bullet}(X_{\Dol}, \Oo_{X_{\Dol}}) \cong H^{\bullet}_{\Dol}( \Oo_X, 0 )$. Moreover, \eqref{eq dol cohom} can be extended to families. 
\end{proposition}

\subsection{de Rham shape and flat connections}
\label{se flat connections}
Let $X$ be a smooth projective variety with a locally free sheaf $E$. A connection is morphism of sheaves of sets $\nabla: E \to E \otimes \Omega_X^1$ that satisfies the Leibniz rule. A connection is flat if $\nabla^2 = \nabla \circ \nabla = 0$, so a flat connection defines a de Rham complex
\[
\Oo_X \xrightarrow{\nabla} E\xrightarrow{\nabla} E \otimes \Omega_X^1 \xrightarrow{\nabla} E \otimes (\Omega_X^1 \wedge \Omega_X^1) \to \cdots 
\]  
The cohomology $H_{\dR}^{\bullet}(E,\nabla)$ of this complex is a sheaf-valued de Rham cohomology introduced by Grothendieck in a letter to Atiyah \cite{grothendieck_atiyah_letter}. The classical construction of de Rham cohomology on $X$ coincides with $H_{\dR}^{\bullet}(\Oo_X, \d)$, where $\d$ is the de Rham differential. In this way, flat connections gives rise to a sheaf-valued de Rham cohomology. 

Given a flat connection $\nabla: E \to E \otimes \Omega_X^1$, the dual morphism $T_X \otimes E \to E$ is a module over the algebra of differential operators $\Dd_X$. This leads us to the notion of a \textit{$\Dd_X$-module}; a module over $\Dd_X$, where also $\Ee$ is quasi-coherent as an $\Oo_X$-module. These are the objects of the category $\QMod_{\Dd_X}(X)$. 

Introduced and studied by Simpson \cite{simpson_dolbeault, simpson_hodge}, the \textit{de Rham shape} $X_{\dR}$ is defined to be the formal completion of the diagonal embedding $\delta : X \hookrightarrow X \times X$. As a category fibered in groupoids, it has families over $S \in \Schemes$ given by $X_{\dR}(S) = \Hom(S_{red}, X)$. The relation between $X_{\dR}$ and flat connections is given by the equivalences of categories 
\begin{equation}
\label{eq crystals and d-modules}
\QCoh(X_{\dR}) \cong \QMod_{\Dd_X}(X),
\end{equation}
as shown by Simpson \cite{simpson_hodge_2} and Porta-Sala \cite[Proposition 4.1.1.]{porta&sala}. The equivalences are compatible with both the subcategories of coherent sheaves and with forgetful functors to $\QCoh(X)$. The relation between $X_{\dR}$ and de Rham cohomology is given by an isomorphism 
\[
R\Gamma( X_{\dR} , \Oo_{X_{\dR}}) \cong H^{\bullet}_{\dR}(X) = H^{\bullet}_{\dR}(\Oo_X, \d).
\]
(e.g. \cite[§6.5.1]{gaitsgory&rozenblyum2}, \cite[Corollary 4.1.2]{porta&sala}). It follows that sheaves on $X_{\dR}$ encode a sheaf-valued de Rham cohomology on $X$, and this is the point of view taken by Simpson in the study of mapping stacks 
\[
\Loc_G(X) = \Maps(X_{\dR}, BG),
\]
the moduli stack that parameterises families of flat connections on $X \times S$, which are sheaf morphisms $\nabla : \Ee \to \Ee \otimes \Omega^1_{X \times S / S}$. 

\subsection{Hodge shape and $\lambda$-connections}
\label{se lambda connections}
Following suggestions of Deligne, the study of $\lambda$-\textit{connections} for formal parameter $\lambda \in \AA^1$ was initiated by Simpson \cite{simpson1, simpson2, simpson_hodge_2} as a way to interpolate between flat connections at $\lambda=1$ and Higgs fields at $\lambda=0$. The notion of a $\lambda$-connection is defined to be a sheaf morphism $\nabla^{\lambda} : \Ee \to \Ee \otimes \Omega^1_X$ that satisfies a $\lambda$-rescaled version of the Leibniz rule: 
\begin{equation}
\label{eq rescalaing leibniz}
\nabla^{\lambda}(f\sigma) = \lambda \cdot f \otimes d\sigma + f \nabla(\sigma).
\end{equation}
At $\lambda = 1$ this is the standard Leibniz rule and at $\lambda = 0$ this is $\Oo$-linearity. There exists a natural rescaling $\GG_m$-action $\nabla^{\lambda} \mapsto \mu \cdot \nabla^{\lambda}$, through which the Leibniz rule transforms as 
\begin{equation}
\label{eq rescale lambda}
\mu \cdot \nabla^{\lambda}(f\sigma) = \mu \lambda \cdot f \otimes d\sigma + f \mu \cdot \nabla^{\lambda}(\sigma),
\end{equation}
so $\mu \cdot \nabla^{\lambda}$ is a $(\mu \lambda)$-connection. Consequently for $\lambda \neq 0$, a $\lambda$-connection can be rescaled to a $1 = (\lambda^{-1} \lambda)$-connection. Higgs fields can then be viewed as a degeneration of a $\GG_m$-family in the limit $\lambda \to 0$. 

With this point of view, Simpson \cite[§7]{simpson_hodge} introduced and studied the \textit{Hodge shape} $X_{\Hod}$, defined as the deformation to the normal bundle of the map $X \to X_{\dR}$ (see \cite[§9.2]{gaitsgory&rozenblyum} for such deformations). The construction comes equipped with a natural projection to the deformation parameter: $\Lambda : X_{\Hod} \to \AA^1$, with preimages
\begin{equation}
\label{eq hodge shape projection}
\begin{split}
& X_{\Hod} \times_{\AA^1} \{0\} \cong X_{\Dol}, \\ 
& X_{\Hod} \times_{\AA^1} \{\lambda\} \cong X_{\dR}, \quad \text{ for all } \lambda \neq 0, \\
\Triv_{\Sigma} : \, \, & X_{\Hod} \times_{\AA^1} (\AA^1 - 0) \xrightarrow{\cong} X_{\dR} \times (\AA^1 - 0 ).
\end{split}
\end{equation}
Let $X$ be a smooth projective variety and $G$ be a reductive group. Relative to $X_{\Hod} \to \AA^1$, we define the \textit{Hodge stack} to be 
\[
\Hodge_G(X) = \Maps(X_{\Hod} / \AA^1, BG),
\]
the moduli stack that parameterises families of $\lambda$-connections over $S \to \AA^1$ \cite[§7]{simpson_hodge}, \cite[§6]{porta&sala}. Recording the $\lambda$-value of a $\lambda$-connection defines a surjection $\Lambda : \Hodge_G(X) \to \AA^1$ with fibers
\begin{equation}
\label{eq triv rescaling}
\begin{split}
& \Hodge_G(X) \times_{\AA^1} \{0\} \cong \Higgs_G(X), \\
& \Hodge_G(X) \times_{\AA^1} \{\lambda\} \cong \Loc_G(X), \quad \text{for every $\lambda \neq 0$}, \\
& \Triv : \Hodge_G(X) \times_{\AA^1} (\AA^1 - 0) \xrightarrow{\cong} \Loc_G(X) \times (\AA^1 - 0), \\ 
&  \hspace{4.1cm}(P,\nabla^{\lambda}) \mapsto (P, \lambda^{-1} \cdot \nabla^{\lambda}, \lambda).
\end{split}
\end{equation}

\subsection{Betti shape and representations}
\label{se betti shape}
We conclude this section with the Betti part of higher non-abelian Hodge theory. On a scheme $X$ with underlying topological space $X^{top}$, Simpson \cite[§2.3]{simpson_dolbeault} defined the \textit{Betti shape} $X_{B}$ to be the constant stack associated to $X^{top}$. Given a reductive group $G$, we then define the \textit{Betti stack} to be  
\[
\Rep_G(X) = \Maps(X_B, BG), 
\]
the moduli stack that parameterises families of representations $\pi_1(X) \to G$ \cite[Remark 3.1.2]{porta&sala}. 

\section{Twistor space and the Deligne stack}
\label{se Deligne stack}

This section recalls the \textit{Deligne-Hitchin twistor space} and generalises the construction to the setting of derived analytic stacks. This results in the construction of the \textit{Deligne stack}. The twistor space constructions use the moduli spaces obtained by introducing Geiseker stability conditions for Higgs bundles, flat connections, $\lambda$-connections and holonomy representations. We recall the mapping stacks from the previous section and adopt the nomenclature introduced by Simpson \cite{simpson1, simpson2}:

\begin{center}
\begin{tabular}{c|c|c}
Simpson shapes $X_{\Sim}$ as &
Derived moduli stacks of & 
Moduli spaces of  
\\
classical 1-stacks &
the form $\Maps(X_{\Sim} , BG)$ 
& semistable objects \\
\hline
$X_{\Dol}$ & $\Higgs_G(X)$ & $\Mm_{\Dol}(X,G)$ \\
$X_{\dR}$  & $\Loc_G(X)$ & $\Mm_{\dR}(X,G)$ \\
$X_{\Hod}$ & $\Hodge_G(X)$ & $\Mm_{\Hod}(X,G)$ \\
$X_{B}$    & $\Rep_G(X)$ & $\Mm_{B}(X,G)$ \\
\end{tabular} 
\end{center}

\subsection{The Deligne-Hitchin twistor space}
\label{se construction}
Twistor theory \cite{penrose, ahs, hklr, salamon} packages the data of a hyperk\"ahler manifold $(M, g, J_1, J_2, J_3)$ into the holomorphic geometry of $\Tw(M)$, the \textit{twistor space} of $M$. For $M=\Mm_{\Dol}(X,G)$, the hyperk\"ahler structures arises from the non-abelian Hodge theorem of Corlette, Donaldson, Hitchin and Simpson \cite{corlette, donaldson, hitchin_self, simpson_higgs}, which constructs an isomorphism
\[
\NAH_1 : \Mm_{\Dol}(X,G) \xrightarrow{\cong} \Mm_{\dR}(X,G).
\]
Take $J_1 = J_{\Dol}$ to be the natural holomorphic structure on $\Mm_{\Dol}(\Sigma,G)$ and pullback the holomorphic structure $J_{\dR}$ on $\Mm_{\dR}(\Sigma,G)$ to define $J_2 := \NAH_1^{*}J_{\dR}$. Then, $J_1 \neq J_2$ is a consequence of the fact that $\NAH_1$ is real-analytic and non-holomorphic map. The remaining complex structure is $J_3 = J_1 J_2$, making up the hyperk\"ahler structure $(\Mm_{\Dol}(X,G), g_{\Hit}, J_1, J_2, J_3)$, where $g_{\Hit}$ is the Hitchin metric constructed by Hitchin \cite{hitchin_self}. 

Following suggestions of Deligne, the \textit{Deligne-Hitchin twistor space} $\Tw(\Mm_{\Dol}(X,G))$ corresponding to the hyperk\"ahler structure on $\Mm_{\Dol}(X,G)$ has been constructed by Simpson \cite[§4]{simpson_hodge_2}. As we now recall, the construction of $\Tw(\Mm_{\Dol}(X,G))$ is complex analytic, so it is more accurate to specify the base curve to be $\Sigma := X^{an}$, the underlying analytic curve. Associating the holonomy representation to a flat connection, the Riemann-Hilbert correspondence of Deligne and Simpson gives rise to a diffeomorphism 
\[
RH : \Mm_{B}(\Sigma,G) \xrightarrow{\cong} \Mm_{\dR}(\Sigma,G).
\]
Consider $\Mm_{\Hod}(\Sigma,G)$ over $\AA^1$ via the natural projection to $\lambda \in \AA^1$. The natural $\GG_m$-action on $\Mm_{\Hod}(\Sigma,G)$ that rescales $\lambda$-connections induces a diffeomorphism 
\[
\Triv : \Mm_{\Hod}(\Sigma,G) \times_{\AA^1} (\AA^1 - 0) \to \Mm_{\dR}(\Sigma,G) \times (\AA^1 - 0),
\]
\[
\quad (P, \nabla^{\lambda}) \mapsto \Big((P, \lambda^{-1} \cdot \nabla^{\lambda}) ~ , ~ \lambda \Big).
\]
We therefore obtain an analytic embedding 
\[
\imath : \Mm_{B}(\Sigma,G) \times (\AA^1 - 0) 
\xrightarrow{\Triv^{-1} \circ (RH \times \id)}
\Mm_{\Hod}(\Sigma,G) \times_{\AA^1} (\AA^1 - 0)
\hookrightarrow
\Mm_{\Hod}(\Sigma,G).
\]
Over $\ol{\Sigma}$, the complex analytic curve with complex conjugate holomorphic structure, one we denote the corresponding analytic embedding by $\overline{\imath}$. The Betti moduli space depends only on the underlying topology, so since $\Sigma^{top} = (\ol{\Sigma})^{top}$, we have $\Mm_{B}(\Sigma,G) = \Mm_{B}(\ol{\Sigma} , G)$. This allows us to define the following gluing.  

\begin{definition}(\cite[§4]{simpson_hodge_2})
\label{de twistor space}.
Let $\Sigma$ be a complex analytic curve and $G$ a reductive structure group. The corresponding \textit{Deligne-Hitchin twistor space} $\Tw(\Mm_{\Dol}(\Sigma,G))$ is the complex analytic manifold given by the following pushout square:
\begin{equation}
\label{eq define glu moduli}
\begin{tikzcd}
\Mm_{B}(\Sigma , G) \times (\AA^1 - 0) = \Mm_{B}(\ol{\Sigma} , G) \times (\AA^1 - 0)
\arrow[r, hook, "\ol\imath"]
\arrow[d, hook, "\imath"']
\arrow[dr, phantom, "\square"]
& \Mm_{\Hod}(\ol{\Sigma},G) \arrow[d]\\
\Mm_{\Hod}(\Sigma,G) \arrow[r] & \Tw(\Mm_{\Dol}(\Sigma,G))
\end{tikzcd}.
\end{equation}
\end{definition}
The pushout \eqref{eq define glu moduli} covers the transition function $(\lambda \mapsto \lambda^{-1}) \in \Aut(\AA^1 - 0)$ over the equator of the Riemann sphere $\PP^1 = \CC \cup \CC$. Consequently, the projections $\Lambda : \Mm_{\Hod}(\Sigma, G) \to \AA^1$ and $\ol{\Lambda} :  \Mm_{\Hod}(\ol{\Sigma}, G) \to \AA^1$ that record the $\lambda$-value of a $\lambda$-connection glue to define 
\[
\tau : \Tw(\Mm_{\Dol}(\Sigma,G)) \to \PP^1.
\]
Let us highlight the role of non-abelian Hodge theory in the twistor theory. Composing $
\NAH_{1} : \Mm_{\Dol}(\Sigma,G) \xrightarrow{\cong} \Mm_{\dR}(\Sigma,G)$ with the $\GG_m$-action defines a diffeomorphism 
\[
\Mm_{\Dol}(\Sigma,G) \times \AA^1 \xrightarrow{\cong} \Mm_{\Hod}(\Sigma,G), 
\quad ( (E,\phi) , \lambda ) \mapsto  
\Big \{ 
\begin{array}{l}
(E,\phi) \quad \text{ when } \lambda = 0 \\
\lambda \cdot \NAH_{1}(E,\phi) \text{ when } \lambda \neq 0
\end{array} 
\]
This glues to the corresponding map over $\ol{\Sigma}$ to define a homeomorphism 
\begin{equation}
\label{eq NAH}
\NAH : \Mm_{\Dol}(\Sigma,G) \times \PP^1 \xrightarrow{\cong} \Tw(\Mm_{\Dol}(\Sigma,G)). 
\end{equation}
With the contraction 
\[
\NAH_{\lambda} := \NAH( \, \cdot \, , \lambda )  : \Mm_{\Dol}(\Sigma,G) 
\xrightarrow{\cong}
\Tw(\Mm_{\Dol}(\Sigma,G)) \times_{\PP^1} \{\lambda\},
\]
one can define a holomorphic structure $J_{\lambda}$ on $\Mm_{\Dol}(\Sigma,G)$ by pullback along $\NAH_{\lambda}$. Consequently one varies the complex structure as one rotates $\lambda \in \PP^1$, giving rise to \textit{hyperk\"ahler rotation}. This idea shall be central to our analysis of hyperholomorphic sheaves in Section \ref{se twistoral BBB}. 

\subsection{Construction of the Deligne stack}
\label{se deligne stack}
We now emulate the construction of the Deligne-Hitchin twistor space in the category of derived analytic stacks. This idea, previously alluded to by Simpson \cite[pg. 55]{simpson_hodge}, invokes the theory of analytic stacks developed by Lurie and Porta \cite{lurie, porta_GAGA, porta_square}. In particular, Porta studies the theory of Simpson shapes in the complex analytic topology \cite{porta}, which, applied to $\Sigma=X^{an}$, allows us to study the analytic mapping stacks
\[
\Higgs_G(\Sigma) = \anMaps(\Sigma_{\Dol}, BG),
\]
\[
\Loc_G(\Sigma) = \anMaps(\Sigma_{\dR},BG),
\]
\[
\Hodge_G(\Sigma) = \anMaps(\Sigma_{\Hod} / \AA^1,BG),
\]
\[
\Rep_G(\Sigma) = \anMaps(\Sigma_{B},BG).
\]
We require a Riemann-Hilbert correspondence that applies in this setting. Inspired by ideas of Simpson \cite{simpson1, simpson2}, this has been achieved by Porta \cite[Theorem 2]{porta} and Holstein-Porta \cite[Theorem 1.5]{holstein&porta}, who construct a remarkable morphism $\nu_{RH} : \Sigma_{\dR} \to \Sigma_{B}$ and study the equivalence 
\[
RH := \nu_{RH}^{*}: \Rep_G(\Sigma) \xlongrightarrow{\cong} \Loc_G(\Sigma).
\]
By rescaling $\lambda$-connections with the natural $\GG_m$-action, one has the equivalence $\Triv$ from \eqref{eq triv rescaling} and the following inverse, which acts as $(\nabla, \lambda) \mapsto \lambda \cdot \nabla$:
\[
\Triv^{-1} : \Loc_G(\Sigma) \times (\AA^1 - 0) \to \Hodge_G(\Sigma) \times_{\AA^1} (\AA^1 - 0).
\]
Let $\ol{\Sigma}$ denote the complex analytic curve with complex conjugate holomorphic structure. The corresponding morphisms shall be denoted by $\ol{RH}$ and $\ol{\Triv}^{-1}$. The Betti shape depends only on the underlying topology, so one has $\Sigma_B = \ol{\Sigma}_B$ and $\Rep_G(\Sigma) = \Rep_G(\ol{\Sigma})$. Placing all of these equivalences together allows us to define a pair of immersions 
\[
\begin{tikzcd}
\Rep_G(\ol{\Sigma}) \times (\AA^1 - 0) \quad
\arrow[r, "\ol{\Triv}^{-1} \circ (\ol{RH} \times \id)"]
& 
\quad \Hodge_G(\ol{\Sigma}) \times_{\AA^1} (\AA^1 - 0)
\arrow[r, hook]
&
\Hodge_G(\ol{\Sigma}) 
\\
\Rep_G(\Sigma) \times (\AA^1 - 0) \quad
\arrow[r, "\Triv^{-1} \circ (RH \times \id)"]
\arrow[u, equal]
& 
\quad \Hodge_G(\Sigma) \times_{\AA^1} (\AA^1 - 0)
\arrow[r, hook]
&
\Hodge_G(\Sigma)
\end{tikzcd}
\]
These immersions fit into the following pushout square, where the pushout is once more a derived analytic stack by a theorem of Porta-Yu \cite[Theorem 1.6]{porta&yu_rep}. 
\begin{definition}
\label{de deligne stack}
Let $\Sigma$ be a complex analytic curve and let $G$ be a reductive group. The corresponding \emph{Deligne stack}, $\Deligne_G(\Sigma)$, is the complex analytic stack defined by the pushout  
\begin{equation}
\label{eq define glu}
\begin{tikzcd}[column sep = huge]
\Rep_G(\Sigma) \times (\AA^1 - 0) = \Rep_G(\ol{\Sigma}) \times (\AA^1 - 0) 
\arrow[r, hook] 
\arrow[d, hook] 
\arrow[dr, phantom, "\square"] 
& \Hodge_G(\ol\Sigma)    
\arrow[d, "\imath_{\ol{\Hod}}"]
\\
\Hodge_G(\Sigma) 
\arrow[r, "\imath_{\Hod}"']
&
\Deligne_G(\Sigma) 
\end{tikzcd}
\end{equation}
This pushout shall be denoted
\[\Deligne_G(\Sigma) = \Hodge_G(\Sigma) \cup_{\Rep} \Hodge_G(\ol{\Sigma}). 
\] 
\end{definition}
The surjections $\Lambda: \Hodge_G(\Sigma) \to \AA^1$ and $\ol{\Lambda}: \Hodge_G(\ol{\Sigma}) \to \AA^1$ that record the $\lambda$-value of a $\lambda$-connection are related by the antipodal map on the Riemann sphere $\PP^1 = \AA^1 \cup \AA^1$. One therefore obtains a pushout surjection   
\begin{equation}
\label{eq define tau}
\tau := \Lambda \cup_{\Rep} \ol \Lambda : \Deligne_G(\Sigma) \to \PP^1,
\end{equation}
such that $\Hodge_G(\Sigma)$ is the preimage of $\PP^1 - \infty$ and $\Hodge_G(\ol{\Sigma})$ is the preimage of $\PP^1 - 0$. Moreover, $\Higgs_G(\Sigma)$, $\Loc_G(\Sigma)$ and $\Higgs_G(\ol{\Sigma})$ are the fibers at $0, 1$ and $\infty$, respectively. 

Let $t_0 : \dSt \to \St$ denote the truncation functor that records the underlying classical 1-stack. Truncation is compatible both with colimits and the derived Riemann-Hilbert correspondence (see \cite[Theorem B]{porta}), so the pushout square \eqref{eq define glu} truncates to a pushout
\begin{equation}
\label{eq glue truncations}
t_0( \Deligne_G(\Sigma) ) \cong t_0(\Hodge_G(\Sigma)) \cup_{\Rep} t_0(\Hodge_G(\ol\Sigma)).
\end{equation}
The 1-stacks give rise to \textit{good moduli spaces} in the language of Alper \cite{alper_good}. With the open substacks 
\[
\Uu := t_0(\Hodge_G(\Sigma))^{\sst} \subset t_0(\Hodge_G(\Sigma)), \quad \ol{\Uu} := t_0(\Hodge_G(\ol{\Sigma}))^{\sst} \subset t_0(\Hodge_G(\ol{\Sigma})),
\]
consisting of Geiseker semistable objects, standard geometric invariant theory constructions result in good quotients, which define good moduli spaces 
\[
f : \Uu \to \Mm_{\Hod}(\Sigma, G), \quad \ol{f} : \ol{\Uu} \to \Mm_{\Hod}(\ol{\Sigma}, G). 
\]
As well as truncation, the gluing \eqref{eq glue truncations} is compatible with restriction to semistable objects, so we obtain a second sub-pushout 
\[
t_0( \Deligne_G(\Sigma) )^{\sst} := \Uu \cup_{\Rep} \ol{\Uu}.
\]
\begin{proposition}
\label{pr deligne moduli}
The gluing $f_{\Del} := f \cup \ol{f} :  t^0(\Deligne_G(\Sigma))^{\sst} \to \Tw(\Mm_{\Dol}(\Sigma, G))$ is a good moduli space. 
\end{proposition}
\begin{proof}
By a criterion of Alper \cite[Proposition 7.9]{alper_good}, the claim is equivalent to checking that 
\[
f^{-1} \circ f (\Uu \cap \ol{\Uu}) = \Uu \cap \ol{\Uu} \quad \text{and} \quad \ol{f}^{-1} \circ \ol{f} (\Uu \cap \ol{\Uu}) = \Uu \cap \ol{\Uu}.
\]
This is an immediate consequence of the facts that $\Uu \cap \ol{\Uu}$ is the restriction to $\PP^1 - \{0, \infty\}$ and that there exists commutative triangles 
\[
\begin{tikzcd}
\Uu \arrow[d, "f"'] \arrow[r, "\Lambda"] & \PP^1 - \infty \\
\Mm_{\Hod}(\Sigma, G) \arrow[ur]
\end{tikzcd}, 
\quad 
\begin{tikzcd}
\ol{\Uu} \arrow[d, "\ol{f}"'] \arrow[r, "\ol{\Lambda}"] & \PP^1 - 0 \\
\Mm_{\Hod}(\ol{\Sigma}, G) \arrow[ur]
\end{tikzcd}.
\qedhere
\]
\end{proof}

We conclude this section with some speculations on the geometry of the Deligne stack. 
\begin{remark} There does not yet exist a notion of a hyperk\"ahler and twistor geometry in the realm of stacks, although see the tentative definition of Katzarkov-Pandit-Spaide \cite[Definition 6.3]{KPS} and the research program described therein\footnote{We thank Pranav Pandit for discussions on the possibility of derived hyperk\"ahler geometry.}. It is consequently tempting to refer to $\Deligne_G(\Sigma)$ as the "twistor stack" of $\Higgs_G(\Sigma)$, which we hope is confirmed in future research. 

It would also be interesting to explore the existence of a \textit{Deligne shape}\footnote{We note that in the terminology of Porta-Sala \cite{porta&sala}, the \textit{Deligne shape} is what we call the \textit{Hodge shape}.} $\Sigma_{\Del}$, defined in terms of Simpson's higher non-abelian Hodge theory, that results in a mapping stack structure 
\[
\Deligne_G(\Sigma) \cong \Maps(\Sigma_{\Del}/\PP^1, BG).
\]   
\end{remark}

\section{The Dirac-Higgs complex}
\label{se dirac-higgs}

This section begins by recalling Hitchin's analytic construction of the Dirac-Higgs bundle on $\Mm_{\Dol}(X,G)$ and the subsequent algebraic intepretation of Hausel. We show how Hausel's interpretation defines a complex on $\Higgs_G(X)$ that we name the \textit{Dirac-Higgs complex}. We then compute an isomorphism between the adjoint-valued Dirac-Higgs complex and the tangent complex of $\Higgs_G(X)$.  

\subsection{History and constructions}
\label{se dirac-higgs constructions}
The Dirac-Higgs bundle is a hyperholomorphic vector bundle constructed by Hitchin \cite{hitchin_dirac} in terms of the underlying gauge theory. Let $\Bb$ denote the infinite dimensional affine space consisting of all Higgs bundles. Hitchin constructs a null space bundle $\Ker(\D^{*}) \to \Bb$ for a family of Dirac operators 
\[
\D_{(E,\phi)}^* = \begin{pmatrix} \overline{\partial}_E &  \phi \\ \phi^* & \partial_E \end{pmatrix}
\]
The descent of $\Ker(\D^{*})$ to the moduli space $\Mm_{\Dol}(X,GL_n)$ is controlled by the existence of \'etale-local universal families $\Uu_i \to \Zz_i$ for a sufficiently fine \'etale atlas $\{ \Zz_i \hookrightarrow \Mm_{\Dol}(X,GL_n)\}_{i \in I}$. The universal families exist via a result of Simpson \cite[Theorem 4.7]{simpson2}. One then gets a descent bundle $\Dd_i \to \Zz_i$ with fibers $\Ker(\D^{*}_{(E,\phi)})$, whose dimension is $2\rank(E)(g(X) - 1)$ by the Atiyah-Singer Index Theorem (see \cite[Proposition 2.1.10]{blaavand}). In particular the rank is constant, and the family of locally free sheaves $\{ \Dd_i \to \Zz_i \}_{i \in I}$ is called the \textit{Dirac-Higgs bundle}. 

An algebraic interpretation of the Dirac-Higgs bundle is constructed by Hausel \cite{hausel} in terms of the projection $p_2 : X \times \Zz_i \to \Zz_i$ and the derived pushforward $R^1p_{2,*}\Uu_i$. The relation between these constructions is the existence of an analytic isomorphism 
\begin{equation}
\label{eq hausel}
\Dd_i \cong (R^1p_{2,*}\Uu_i)^{an},
\end{equation}
that arises from Hodge theory \cite{blaavand, hausel}. The \'etale-local perspective one is forced to take in the above constructions comes at the price of introducing a gerbe, defined by cocycles over $\Zz_i \cap \Zz_j$. The Dirac-Higgs bundle twisted by this gerbe is studied by Blaavand \cite{blaavand} and the first author alongside Jardim \cite{franco&jardim}. In the case $G = GL_1$, one happens to have a global universal family, where the corresponding global Dirac-Higgs bundle is studied by Bonsdorff \cite{bonsdorff_1} and Frejlich-Jardim \cite{frejlich&jardim}. 

These constructions indicate two reasons why the Dirac-Higgs bundle should live on the stack $\Higgs_G(X)$ rather than the moduli space $\Mm_{\Dol}(X,G)$. The first is that Hitchin's initial constructions are performed on the underlying gauge theory $\Bb$, the space of all $\CC$-valued points on $\Higgs_G(X)$. One could consider relative Dirac operators adapted to general S-valued points on $\Higgs_G(X)$, giving rise to a family of Dirac operators. Rather than take this route, we take note of the second reason for working on the stack: a globally defined universal family on $\Higgs_G(X)$ always exists, or in other words, the gerbe from the above constructions trivialises after passing to $\Higgs_G(X)$.  

With this in mind, our Dirac-Higgs constructions on $\Higgs_G(X)$ are as follows. Let $\Uu = (\Ff, \Theta)$ denote the universal $G$-Higgs bundle, which consists of a $G$-bundle $\Ff \to X \times \Higgs_G(X)$ and a relative Higgs field $\Theta \in H^0(X, \ad(\Ff) \otimes p_1^{*}K_X)$. To study general reductive group $G$ we introduce a representation $\rho : G \to GL_n$. Denote by $\rho(\Ff)$ the associated vector bundle and by $\rho(\Theta) \in H^0(X, \End(\rho(\Ff) \otimes p_1^{*}K_X)$ the associated section. We interpret the pair $\rho(\Uu) = (\rho(\Ff), \rho(\Theta))$ as a two term complex 
\[
\left[ \rho(\Ff) \xrightarrow{\rho(\Theta)} \rho(\Ff) \otimes p_1^{*}K_X \right] \in \Perf(X \times \Higgs_G(X)).
\]
Then, with the projection $X \times \Higgs_G(X) \xrightarrow{p_2} \Higgs_G(X)$, one has the following stack theoretic generalisation of the Dirac-Higgs bundle.
\begin{definition}
\label{de dirac-higgs}
The \emph{($\rho$-valued) Dirac-Higgs complex} on $\Higgs_G(X)$ is given by 
\[
\DD_{\rho} = Rp_{2,*} \rho(\Uu) \in \Perf( \Higgs_G(X) ), 
\]
and the \textit{($\rho$-valued) Dirac-Higgs sheaf} is the degree-one cohomology sheaf $H^1(\DD_{\rho})$.
\end{definition}

After truncation to the classical 1-stack $t_0(\Higgs_G(X))$ and restriction to the semistable locus, the Dirac-Higgs complex descends as expected:

\begin{proposition}
Let $G=GL_n$ and let $\rho =\id$ be the standard representation. The Dirac-Higgs sheaf $H^1(\DD_{\id})$ descends to the Dirac-Higgs bundle on the moduli space 
\[
f_{\Dol} : t_0(\Higgs_G(X))^{\sst} \to \Mm_{\Dol}(X,G).
\]
\end{proposition}
\begin{proof}
Given a sufficiently fine \'etale atlas $\phi_i : \Zz_i \to t_0(\Higgs_{GL_n}(X))^{\sst}$, the pullbacks $\phi_i^{*}\Uu$ are \'etale-local universal families for the atlas $f_{\Dol} \circ \phi_i : \Zz_i \to \Mm_{\Dol}(X,GL_n)$. Therefore $\phi_i^{*}\Uu$ can be used to construct the Dirac-Higgs bundle: $\Dd_i = R^1q_{2,*}(\phi_i^{*}\Uu)$, where $q_2 : X \times \Zz_i \to \Zz_i$ is the natural projection. By an application of base change, this is isomorphic to 
\[
\Dd_i \cong H^1(\DD_{\rho})|_{\Zz_i}. \qedhere
\] 
\end{proof}
At a $\CC$-valued point $(P,\theta) \in \Higgs_G(X)$, the Dirac-Higgs complex has fibers
\begin{equation} 
\label{eq fiber Dirac-Higgs complex} 
\DD_\rho | _{(P,\theta)} 
\cong H_{\Dol}^{\bullet} \left( \rho(\Uu) |_{X \times {\{(P,\theta) \} }} \right) 
= H_{\Dol}^{\bullet}( \rho(P,\theta) ), 
\end{equation}
so $\DD_\rho$ is a family of Dolbeault cohomologies. 

\subsection{The adjoint representation}
For the adjoint representation $\rho = \ad : G \to GL(\gG)$, the formula in \eqref{eq fiber Dirac-Higgs complex} yields the deformation theory of a Higgs bundle on a smooth projective curve. This observation extends to an isomorphism between the tangent complex on $\Higgs_G(X)$ and the adjoint-valued Dirac-Higgs complex.

\begin{proposition}
\label{pr tangent complex and Dirac-Higgs}
The tangent complex $\TT_{\Higgs}$ of $\Higgs_G(X) = \Maps(X_{\Dol}, BG)$ is isomorphic to
\[
\TT_{\Higgs} \cong Rq_{2,*} \ad(\Vv_{\Dol})[1] 
\cong Rp_{2,*} \ad(\Uu)[1].
\]
In particular, $\TT_{\Higgs}$ is isomorphic to the adjoint-valued Dirac-Higgs complex $\DD_{\ad}:= Rp_{2,*} \ad(\Uu)[1]$. 
\end{proposition}

\begin{proof}
Let $\gamma$ denote the quotient map $\Spec(\CC) \to BG$, the universal $G$-bundle on $BG$. The tangent complex of $BG$ is given by the 1-shifted vector space $\TT_{BG} = \gG[1] \cong \ad(\gamma)[1]$. Let $\ev : X_{\Dol} \times \Maps( X_{\Dol} , BG)$ denote the evaluation map. Then, $\gamma$ pulls back to $\Vv_{\Dol} = \ev^{*}\gamma$. Consider the commutative diagram 
\begin{equation*}
\begin{tikzcd}[column sep = huge]
BG & X_{\Dol} \times \Maps( X_{\Dol} , BG) \arrow[l, "\ev"'] \arrow[r, "q_2"] 
&  \Maps( X_{\Dol} , BG ) \\
&  X \times \Maps(X_{\Dol} , BG) \arrow[ur, "p_2"'] \arrow[u, "\kappa \times 1"']
& 
\end{tikzcd}.
\end{equation*}
A formula of To\"en \cite[§3.4]{toen} computes the tangent complex of a mapping stack, first in the case where the source stack is a smooth scheme, but also see the subsequent discussion for the statement that applies to $X_{\Dol}$. To\"en's formula provides the first isomorphism:
\begin{align*} 
\TT_{\Higgs} & \cong Rq_{2,*} \circ \ev^{*} \TT_{BG} \\
& \cong Rq_{2,*} \ad(\ev^{*}\gamma)[1] \\
& \cong Rq_{2,*} \ad(\Vv_{\Dol})[1].
\end{align*}
The second isomorphism follows from the formula $(\kappa\times \id)^{\heart} \rho(\Vv_{\Dol}) \cong \rho(\Uu)$ from Corollary \ref{co universal higgs and bun} and the moreover statement in Proposition \ref{pr dol cohomology}, which provides an equivalence of functors
\[
Rp_{2,*} \simeq Rq_{2,*} (\kappa \times \id)_{\heart}.
\]
One can therefore write 
\begin{align*}
Rq_{2,*} \ad(\Vv_{\Dol})[1] 
& \cong Rq_{2,*} (\kappa \times \id)_{\heart}(\kappa \times \id)^{\heart} \ad(\Vv_{\Dol})[1] \\
& \cong Rp_{2,*} \ad(\Uu)[1]. \qedhere
\end{align*}
\end{proof}

\section{Integral functors in non-abelian Hodge theory}
\label{se integral functors}
Integral functors are enjoyed by the derived categories of all geometric objects. In derived algebraic geometry, they play an extra special role thanks to powerful representability theorems \cite{BZFN, BZNP}. It is therefore natural to rewrite the Dirac-Higgs complex within an integral functor formalism. On the trivial Higgs bundle $(\Oo_X, 0)$, one simply adds the trivial twist:
\[
\DD_{\rho} = Rp_{2,*} (\rho(\Uu) \otimes p_1^{*}(\Oo_X, 0) ),
\]
thus evaluating an integral functor $\Hgs(X,K_X) \to \Perf(\Higgs_G(X))$ with universal integral kernel. This section is dedicated to the construction and study of similar integral functors in the context of non-abelian Hodge theory that we name \textit{Dolbeault, de Rham, Hodge} and \textit{Betti} functors. Finally, we construct a \textit{Deligne functor} as a gluing of complex conjugate analytic Hodge functors. 

\subsection{Construction of integral functors}
\label{se bonsdorff}
\label{se integral functors and GL}
\label{se exchange hecke wilson}

Our functors are parameterised by a formal variable\footnote{The notation $\Sim$ is chosen in homage to Carlos Simpson's foundational papers on higher non-abelian Hodge theory \cite{simpson_hodge_2, simpson_dolbeault, simpson_dolbeault_2,simpson_hodge}.} $\Sim \in \{ \Dol, \dR, \Hod, B \}$ and a representation $\rho : G \to GL_n$. It will be convenient to introduce the notation $\AA^{\Hod} = \AA^1$ and $\AA^{\Dol} = \AA^{\dR} = \AA^{B} = \{ 0 \}$. We define integral functors with kernel taken to be the associated vector bundle of the universal $G$-bundles $\Vv_{\Sim} \to X_{\Sim} \times_{\AA^{\Sim}} \Maps(X_{\Sim}, BG)$.  

We define our functors on categories $\Perf(\cdot) \subset \QC(\cdot)$ of perfect complexes. This is possible by Porta-Sala's notion of a \textit{categorically proper morphism} between derived stacks, which satisfies the coherence property that perfect complexes pushforward to perfect complexes \cite[Proposition 2.3.23]{porta&sala}. Since $X_{\Sim} \to \AA^{\Sim}$ is categorically proper in all cases \cite[Proposition 3.1.1, Proposition 4.1.1, Lemma 5.3.2]{porta&sala}, it follows that pushforward along the projection $f_2$ in the diagram 
\[
X_{\Sim} \xleftarrow{f_1} X_{\Sim} \times_{\AA^{\Sim}} \Bun_G(X_{\Sim}) \xrightarrow{f_2} \Bun_G(X_{\Sim}),
\]
preserves the subcategories $\Perf(\cdot) \subset \QC(\cdot)$ of perfect complexes. 

\begin{definition}
\label{de integral functors}
Let $\Sim \in \{ \Dol, \dR, \Hod, B \}$ and let $\rho : G\to GL_n$ be a representation. The \textit{algebraic Dolbeault, de Rham, Hodge and Betti functors} are  
\[
\Phi_{\Sim,\rho} : \Perf(X_{\Sim}) \to \Perf(\Maps(X_{\Sim}, BG)), \quad \Ee \longmapsto Rf_{2,*}( \rho(\Vv_{\Sim}) \otimes^{L} Lf_1^{*}\Ee ).
\] 
\end{definition}
Our focus here on compact objects $\Perf(\cdot) \subset \QC(\cdot)$ will play a key role in the later construction of (BBB)-branes. In the complex analytic topology, this definition works in exactly the same manner. Given an analytic space $\Sigma$, we recall our notation $\anMaps(\Sigma_{\Sim},BG)$ for the the derived analytic mapping stacks on the analytic Simpson shapes $\Sigma_{\Sim}$. Let $\Vv_{\Sim}^{an} \to \Sigma_{\Sim} \times_{\AA^{\Sim}} \Maps(\Sigma_{\Sim}, BG)$ denote the corresponding universal families and let $f_i^{an}$ denote the natural projections on this product. 

\begin{definition}
\label{de analytic integral functors}
Let $\Sim \in \{ \Dol, \dR, \Hod, \B \}$ and let $\rho : G\to GL_n$ be a representation. The  \textit{analytic Dolbeault, de Rham, Hodge and Betti functors} are 
\[
\Psi_{\Sim,\rho} : \Perf(\Sigma_{\Sim}) \to \Perf(\anMaps(\Sigma_{\Sim},BG)), \quad \Ee \longmapsto R(f_2^{an})_{*}( \rho(\Vv_{\Sim}^{an}) \otimes^{L} Lf_1^{an,*}\Ee ).
\] 
\end{definition}
Next we consider integral functors that are adapted to the abelian categories $\Hgs(Z, L)$, recalling this denotes Higgs sheaves valued in a line bundle $L \to Z$. Given a map $f:Y \to Z$, there exists pullback functors $f^{*} : \Hgs(Z, L) \to \Hgs(Y, f^{*}L)$ that sends a Higgs field $\phi : \Ee \to \Ee \otimes L$ to the pullback $f^{*}\phi : f^{*}\Ee \to f^{*}\Ee \otimes f^{*}L$. In the case where the morphisms considered are the natural projections 
\[
X \xleftarrow{p_1} X \times \Higgs_G(X) \xrightarrow{p_2}\Higgs_G(X), 
\]
one can then \textit{integrate in Dolbeault cohomology} to define a functor 
\[
p_{2,\star} : \Hgs(X \times \Higgs_G(X), p_1^{*}K_X) 
\to \QCoh(\Higgs_G(X)).
\]
We note that perfect and coherent objects are preserved if we assume $X$ to be a smooth projective variety, so that $p_2$ is smooth and proper. We also note that finiteness of Dolbeault cohomology of a Higgs sheaf $(E,\phi)$ is controlled by the usual cohomology of $E$. Finally, we note that Higgs sheaves admit a natural tensor product, acting as
\[
( \Ee_1 \xrightarrow{\phi_1} \Ee_1 \otimes p_1^{*}K_X ) 
\otimes
( \Ee_2 \xrightarrow{\phi_2} \Ee_2 \otimes p_1^{*}K_X ) 
= ( \Ee_1 \otimes \Ee_2 \xrightarrow{\phi_1 \otimes \id + \id \times \phi_2 } \Ee_1 \otimes \Ee_2 \otimes p_1^{*}K_X ).
\]
Let $\DHgs(X,K_X)$ denote the derived category of coherent Higgs sheaves and $\PHgs(X,K_X) \subset \DHgs(X,K_X)$ denote the subcategory of perfect complexes, which inherits derived versions of the functors just described. We now define integral functors, parameterised by a representation $\rho : G \to GL_n$, with integral kernels given by associated bundles of the universal G-Higgs bundle $\Uu \to X \times \Higgs_G(X)$. As for the previous definition, the corresponding objects/morphisms in the analytic topology shall carry the decoration $(\cdot)^{an}$. 
\begin{definition}
\label{de bonsdorff}
Let $\rho : G \to GL_n$ be a representation. The corresponding \emph{algebraic Bonsdorff functor} is given by 
\[
\Phi_{\Bon,\rho} : \PHgs(X ,K_X) \to \Perf(\Higgs_G(X)), 
\]
\[ 
(E,\phi)^\bullet \longmapsto Rp_{2,\star}\Big( \rho(\Uu) \otimes p_1^{*} (E,\phi)^\bullet \Big).
\] 
Similarly, the \emph{analytic Bonsdorff functor} is given by  
\[
\Psi_{\Bon, \rho} : \PHgs(\Sigma , K_\Sigma) \to \Perf(\Higgs_G(\Sigma)), 
\]
\[ 
(E,\phi)^\bullet \longmapsto R(p_{2}^{an})_{\star}\Big( \rho(\Uu^{an}) \otimes p_1^{an, \star} (E, \phi)^\bullet \Big).
\] 
\end{definition}
For $G = GL_1$, a similar functor has been defined and studied by Bonsdorff \cite[Definition 2.3.4]{bonsdorff_1}, where the geometry 
\[
\Higgs_{GL_1}(X) = T^{*}\Pic(X) \cong \Pic(X) \times H^0(X, K_X),
\]
was essential to his description. Returning to our initial motivation for this section, one can now write the Dirac-Higgs complex $\DD_{\rho} := Rp_{2,*}(\rho(\Uu))$ in terms of the Bonsdorff functor as 
\[
\DD_{\rho} = \Phi_{\Bon, \rho}(\Oo_X \xrightarrow{0} K_X ).
\]
This simple observation was one of the initial motivations for this paper, once it was noticed that such a description globalises an \'{e}tale-local observation made by the first author and Jardim \cite[Remark 4.5]{franco&jardim}. 

Next we compare the Bonsdorff and Dolbeault functors. Recalling Proposition \ref{pr QCoh Dol}, one has the equivalence between their domains:
\begin{equation}
\label{eq kappa equiv}
\kappa^{\heart} : \Perf(X_{\Dol}) \xrightarrow{\cong} \PHgs(X,K_X),
\end{equation}
with inverse $\kappa_{\heart}$. Moreover, Corollary \ref{co universal higgs and bun} provides an isomorphism between their integral kernels:
\begin{equation}
\label{eq subst}
(\kappa \times \id)^{\heart} \rho(\Vv_{\Dol}) \cong \rho(\Uu).
\end{equation} 
\begin{proposition}
\label{pr Dol Mod commute}
The algebraic Bonsdorff and Dolbeault functors fit into a commutative diagram 
\begin{equation*}
\begin{tikzcd}[column sep = huge]
\Perf(X_{\Dol}) 
\arrow[dd, "\kappa^{\heart}"'] \arrow[rd, "\Phi_{\Dol, \rho}"]  
\\
& \Perf(\Higgs_G(X)) 
\\
\PHgs(X,K_X)
\arrow[ru, "\Phi_{\Bon, \rho}"']
\end{tikzcd}.
\end{equation*}
Moreover, the same is true in the analytic topology.
\end{proposition}
\begin{proof}
The proof is identical in the algebraic and analytic topologies, so we proceed with the algebraic case. Proposition \ref{pr dol cohomology} states that Dolbeault cohomology of a Higgs bundle satisfies 
\[
H^{\bullet}_{\Dol}( E,\phi ) \simeq H^{\bullet}(X, \kappa_{\heart}( E,\phi ) ),
\]
which moreover can be extended to families. Viewing pushforward is a relative version of cohomology, one has an equivalence of functors 
\[
p_{2,\star} \simeq q_{2,*} \circ (\kappa \times \id)_{\heart}.
\]
By functoriality, similar equivalences hold for pullback and tensor product:
\[
p_1^{*} \kappa^{\heart} \simeq (\kappa \times \id)^{\heart} q_1^{*}, \quad 
\kappa^{\heart} (\cdot) \otimes \kappa^{\heart} (\cdot) \simeq \kappa^{\heart} \Big( (\cdot) \otimes (\cdot) \Big).
\]
Fix $\Ff \in \Perf(X_{\Dol})$. In combination with \eqref{eq subst}, the above equivalences yield
\begin{align*}
\Phi_{\Bon, \rho} (\kappa^{\heart} \Ff^\bullet)
& = Rp_{2,\star} \Big( \rho(\Uu) \otimes p_1^{*} \kappa^{\heart} \Ff^\bullet \Big) \\
& \cong Rp_{2,\star} \Big( (\kappa \times \id)^{\heart} \rho(\Vv_{\Dol}) \otimes \kappa \times \id)^{\heart} q_1^{*} \Ff^\bullet \Big) \\
& \cong Rq_{2,*} (\kappa \times \id)_{\heart} (\kappa \times \id)^{\heart} \Big( \rho(\Vv_{\Dol}) \otimes q_1^{*} \Ff^\bullet \Big) \\
& \cong Rq_{2,*} \Big( \rho(\Vv_{\Dol}) \otimes q_1^{*} \Ff^\bullet \Big). \qedhere
\end{align*}
\end{proof}


We conclude this section by noting the Bonsdorff functor computes Dolbeault cohomology. 
\begin{proposition}
\label{pr bonsdorff fibers}
Fix $(E,\phi) \in \Hgs(X,K_X)$. At a point $f: \Spec(\CC) \to \Higgs_G(X)$ with corresponding $G$-Higgs bundle $(P,\theta) = (\id \times f)^{*}\Uu$ over $X$, the fibers of $\Phi_{\Bon, \rho}(E,\phi)$ are given by
\[
\Phi_{\Bon, \rho}(E,\phi) |_{(P,\theta)} \cong H^{\bullet}_{\Dol}( (E,\phi) \otimes \rho(P,\theta) ).
\]
\end{proposition}

\begin{proof}
By an application of base change there exists an isomorphism of vector spaces 
\begin{align*}
\Phi_{\Bon}(E,\phi)|_{(P,\theta)} 
& = Rp_{2,*} ( \rho(\Uu) \otimes p_1^{*}(E,\phi) )|_{(P,\theta)}
\\
& \cong H^{\bullet}_{\Dol}\Big( \rho(\Uu)|_{X \times \{ (P,\theta) \}} \otimes p_1^{*}(E,\phi) |_{X \times \{ (P,\theta) \}} \Big) \\
& = H^{\bullet}_{\Dol}\Big( \rho(P,\theta) \otimes (E,\phi)  \Big). \qedhere
\end{align*}
\end{proof}

\subsection{The Deligne functor}
We now study the relationship between a pair of analytic Hodge functors and the pushout square from Definition \ref{de deligne stack},
\begin{equation}
\label{eq del again}
\begin{tikzcd} 
\Rep_G(\Sigma) \times (\AA^1 - 0) = \Rep_G(\ol\Sigma) \times (\AA^1 - 0)
\arrow[r, hook]
\arrow[dr, phantom, "\square"] 
\arrow[d, hook]
& 
\Hodge_G(\ol{\Sigma})
\arrow[d, hook]
\\
\Hodge_G(\Sigma)
\arrow[r, hook]
& 
\Deligne_G(\Sigma)
\end{tikzcd}.
\end{equation}
The maps that feature in this pushout are the Riemman-Hilbert correspondence and the rescaling map:  
\[
RH:\Rep_G(\Sigma) \to \Loc_G(\Sigma), \quad \Triv^{-1} : \Loc_G(\Sigma) \times (\AA^1 - 0) \to \Hodge_G(\Sigma) \times_{\AA^1} (\AA^1 - 0). 
\]
Since the property of a complex being \textit{perfect} is local, the above pushout induces the following pullback square. We pay particular attention to specifying the functors involved, as these are used in our subsequent constructions. 
\begin{equation}
\label{eq sheaves on deligne}
\begin{tikzcd} 
\Perf(\Deligne_G(\Sigma)) 
\arrow[r] 
\arrow[d] \arrow[dr, phantom, "\square"] 
& 
\Perf(\Hodge_G(\ol{\Sigma})) 
\arrow[d, "(\ol{RH} \times \id)^{*} \circ \ol{\Triv}_{*} \circ (\cdot)|_{\AA^1 - 0}"] 
\\
\Perf( \Hodge_G(\Sigma)) \quad \quad \quad 
\arrow[r, "(RH \times \id)^{*} \circ \Triv_{*} \circ (\cdot)|_{\AA^1 - 0}"'] 
& 
\quad \quad \quad 
\Perf(\Rep_G(\Sigma) \times (\AA^1 -0)  ) = \Perf(\Rep_G(\ol{\Sigma}) \times (\AA^1 -0) )
\end{tikzcd}.
\end{equation}
It follows that $\widetilde{\Gg} \in \Perf(\Deligne_G(\Sigma))$ can be represented as 
\begin{equation}
\label{eq tilde g}
\widetilde{\Gg} = \Gg \cup_{g} \ol{\Gg},
\end{equation}
consisting of local complexes $\Gg \in \Perf(\Hodge_G(\Sigma))$ and $\ol{\Gg} \in \Perf(\Hodge_G(\Sigma))$ together with a transition function $g: \Gg_{\Rep} \xrightarrow{\cong} \ol{\Gg}_{\Rep}$. Here we have introduced the notation $(\cdot)_{\Rep}$ for  
\begin{equation}  \label{eq definition of Gg_B}
\Gg_{\Rep} := (RH \times \id)^{*} \circ \Triv_{*}(\Gg|_{\AA^1 - 
 0}) \in \Perf(\Rep_G(\Sigma) \times (\AA^1 -0 )), 
\end{equation}
\begin{equation*}
\ol{\Gg}_{\Rep} := (\ol{RH} \times \id)^{*} \circ \ol{\Triv}_{*}( \ol\Gg|_{\AA^1 - 0} ) \in \Perf(\Rep_G(\ol{\Sigma}) \times (\AA^1 - 0 ) ).
\end{equation*}
A similar construction can be performed at the level of the Simpson shapes. With the morphisms 
\[
\nu_{RH} : \Sigma_{\dR} \to \Sigma_B, \quad \Triv_{\Sigma}: \Sigma_{\Hod} \times_{\AA^1} (\AA^1 - 0) \to \Sigma_{\dR} \times (\AA^1 - 0),
\]
we define the Cartesian diagram 
\begin{equation}
\label{eq Dom}
\begin{tikzcd} 
\Perf(\Sigma_{\Hod}) \times_B \Perf(\ol{\Sigma}_{\Hod}) 
\arrow[r] 
\arrow[d] \arrow[dr, phantom, "\square"] 
& 
\Perf(\ol{\Sigma}_{\Hod}) 
\arrow[d, "(\ol{\nu}_{RH} \times \id)_{*} (\ol{\Triv}_{\Sigma, *} \circ (\cdot)|_{\AA^1 - 0}"] 
\\
\Perf(\Sigma_{\Hod}) 
\quad \quad
\arrow[r, " (\nu_{RH} \times \id)_{*} \Triv_{\Sigma, *} \circ (\cdot)|_{\AA^1 - 0}"'] 
& 
\quad \quad 
\Perf(\Sigma_{B} \times (\AA^1 - 0)) = \Perf(\ol{\Sigma}_{B} \times (\AA^1 - 0))
\end{tikzcd}.
\end{equation}
By definition, $\widetilde{\Ff} \in \Perf(\Sigma_{\Hod}) \times_B \Perf(\ol{\Sigma}_{\Hod})$ is given by 
\begin{equation}
\label{eq tilde f}
\widetilde{\Ff} = \Ff \cup_f \ol{\Ff},
\end{equation}
consisting of complexes $\Ff \in \Perf(\Sigma_{\Hod})$ and $\ol{\Ff} \in \Perf(\ol{\Sigma}_{\Hod})$, together with a transition function $f: \Ff_B \xrightarrow{\cong} \ol{\Ff}_B$. Here we have introduced the notation $(\cdot)_{B}$ for
\begin{equation}  
\label{eq definition of Ff_B}
\Ff_B := (\nu_{RH} \times \id)_{*} \circ \Triv_{\Sigma, *} (\Ff|_{\AA^1 - 0} ) \in \Perf(\Sigma_B \times (\AA^1 - 0)),
\end{equation}
\begin{equation*}
\ol{\Ff}_B := (\ol{\nu}_{RH} \times \id)_{*} \circ \ol{\Triv}_{\Sigma, *}( \ol\Ff|_{\AA^1 - 0}  ) \in \Perf(\ol{\Sigma}_B \times (\AA^1 - 0)).
\end{equation*}
With this notation in place, the rest of this section is dedicated to completing the dashed arrow in the following diagram. The front and back squares are the Cartesian squares mentioned above. The diagonal arrows are are pair $\Psi_{\Hod, \rho}$ and $\Psi_{\ol{\Hod}, \rho}$ of analytic Hodge functors, defined over $\Sigma$ and $\ol{\Sigma}$, respectively. $\widehat{\Psi}_{B,\rho}$ is a relative version of the analytic Betti functor, taken over $\AA^1 - 0$. 
\[
\begin{tikzcd}[row sep={50,between origins}, column sep={70,between origins}]
& 
\Perf(\Sigma_{\Hod}) \times_B \Perf(\ol{\Sigma}_{\Hod}) 
\ar{rr}\ar{dd} 
\arrow[dl, dashed, "\Psi_{\Del, \rho}"'] 
& & 
\Perf(\ol{\Sigma}_{\Hod})
\vphantom{\times_{S_1}} \ar{dd} \ar[dl, "\Psi_{\ol{\Hod}, \rho}"]
\\
\Perf(\Deligne_G(\Sigma)) \ar[crossing over]{rr} \ar{dd} & & \Perf(\Hodge_G(\ol{\Sigma}))
\\
& \Perf(\Sigma_{\Hod}) \ar{rr} \ar[dl, "\Psi_{\Hod, \rho}"] & &  \Perf(\Sigma_{B} \times (\AA^1 - 0)) \vphantom{\times_{S_1}} \ar[dl, "\widehat{\Psi}_{B,\rho}"] 
\\
\Perf(\Hodge_G(\Sigma)) \ar{rr} && \Perf(\Rep_G(\Sigma) \times (\AA^1 - 0)) \ar[from=uu,crossing over]
\end{tikzcd}
\]
The main step to completing this diagram is to show that Hodge, de Rham and Betti functors commute with the morphisms that define the Cartesian squares. 

\begin{lemma}
\label{le commutative diagram}
Let $\widehat \Psi_{\dR , \rho}$ and $\widehat \Psi_{B , \rho}$ denote relative versions of the analytic de Rham and Betti functors, relative over $\AA^1 - 0$. These are simply the integral functors with kernels $\widehat{\Vv}_{\dR}$ and $\widehat{\Vv}_{B}$, the extensions of $\Vv_{\dR}$ and $\Vv_{B}$, relative to $\AA^1 - 0$. Then, there exists a commutative diagram  
\begin{equation}
\label{eq integral functors commute}
\begin{tikzcd}[column sep = huge]
\Perf( \Sigma_{\Hod}) 
\arrow[r, "\Psi_{\Hod, \rho}"]
\arrow[d, "\Triv_{\Sigma,*} \circ (\cdot)|_{\AA^1 - 0}"']
& \Perf ( \Hodge_G(\Sigma) ) 
\arrow[d, "\Triv_* \circ (\cdot)|_{\AA^1 - 0}" ] 
\\
\Perf( \Sigma_{\dR} \times (\AA^1 - 0) ) 
\arrow[r, "\widehat{\Psi}_{\dR ,\rho}"] 
\arrow[d, "(\nu_{RH} \times \id)_{*}"']
& \Perf( \Loc_G(\Sigma) \times (\AA^1 - 0) ) 
\arrow[d, "(RH \times \id)^{*}"] 
\\
\Perf( \Sigma_{B} \times (\AA^1 - 0))  
\arrow[r, "\widehat{\Psi}_{B , \rho}"] 
& \Perf( \Rep_G(\Sigma) \times (\AA^1 - 0) ) 
\end{tikzcd}.
\end{equation}
\end{lemma}
We delay the proof until the following section, where it is checked in Lemma \ref{le commutivity of lower diagram} and \ref{le commutivity of upper diagram}. First, let us instead describe the following consequence. 

\begin{proposition}
\label{pr satisfy gluing}
In the notation of \eqref{eq tilde f}, fix an element $\widetilde{\Ff} = \Ff \cup_f \ol\Ff \in \Perf(\Sigma_{\Hod}) \times_B \Perf(\ol{\Sigma}_{\Hod})$. In the notation of \eqref{eq tilde g}, consider   
\[
\Gg := \Psi_{\Hod,\rho}(\Ff) , \quad \ol{\Gg} := \Psi_{\ol{\Hod},\rho}(\ol{\Ff}), \quad g = \widehat{\Psi}_{B, \rho}(f). 
\]
We claim that $\Gg \cup_g \ol{\Gg}$ is a well-defined object of $\Perf(\Deligne_G(\Sigma))$. In other words, we claim there exists a well-defined functor 
\[
\Perf(\Sigma_{\Hod}) \times_B \Perf(\ol{\Sigma}_{\Hod}) \to \Perf(\Deligne_G(\Sigma)),
\]
\[
\Ff \cup_f \ol\Ff \longmapsto \Psi_{\Hod,\rho}(\Ff) \cup_{\widehat{\Psi}_{B, \rho}(f)} \Psi_{\ol{\Hod},\rho}(\ol{\Ff}).
\]
\end{proposition}

\begin{proof}
It suffices to check that $g$ is an isomorphism $\Gg_{\Rep} \cong \ol{\Gg}_{\Rep}$. First we compute $\Gg_{\Rep}$ and $\ol{\Gg}_{\Rep}$. Lemma \ref{le commutative diagram}, applied over $\Sigma$, provides an isomorphism
\begin{align*}
\Gg_{\Rep} & := (RH \times \id)^{*} \circ \Triv_{*}(\Gg|_{\AA^1 - 
 0})
 \\
& \cong \widehat{\Psi}_{B, \rho} \left( (\nu_{RH} \times \id)_{*} \circ \Triv_{\Sigma,*} (\Ff|_{\AA^1 - 0}) \right)  = \widehat{\Psi}_{B, \rho}(\Ff_B).
\end{align*}
The same calculation over $\ol{\Sigma}$ provides
\begin{align*}
\ol{\Gg}_{\Rep} & := (\ol{RH} \times \id)^{*} \circ \ol{\Triv}_{*}( \ol\Gg|_{\AA^1 - 0} ) \\
& \cong \widehat{\Psi}_{\ol{B}, \rho} \left( (\ol{\nu}_{RH} \times \id)_{*} \circ \Triv_{\ol{\Sigma},*} (\ol{\Ff}|_{\AA^1 - 0})\right) = \widehat{\Psi}_{\ol{B}, \rho}(\ol{\Ff}_B).
\end{align*}
Therefore $\widehat{\Psi}_{B, \rho} = \widehat{\Psi}_{\ol{B}, \rho}$ evaluated on $f: \Ff_B \xrightarrow{\cong} \ol\Ff_B$ provides the sought-after isomorphism:
\[
g := \widehat{\Psi}_{B, \rho}(f) : \Gg_{\Rep} \xrightarrow{\cong} \ol\Gg_{\Rep}. \qedhere
\]
\end{proof}


\begin{definition}
\label{de deligne functor}
Let $\Sigma$ be an analytic curve and let $\rho : G \to GL_n$ be a representation. The corresponding \textit{Deligne functor} is defined to be
\[
\Psi_{\Del, \rho} : \Perf(\Sigma_{\Hod}) \times_B \Perf(\ol{\Sigma}_{\Hod}) \to \Perf(\Deligne_G(\Sigma)),
\]
\[
\Ff \cup_f \ol\Ff 
\, \, \longmapsto \, \, 
\Psi_{\Hod,\rho}(\Ff) 
\cup_{\widehat{\Psi}_{B,\rho}(f)} 
\Psi_{\ol{\Hod},\rho}(\ol\Ff),
\]
which is well-defined by Proposition \ref{pr satisfy gluing}.
\end{definition}

\subsection{Compatibilities}

This section is dedicated to the proof of Lemma \ref{le commutative diagram} and thus completes our construction of the Deligne functor. This lemma consists of establishing natural commutative diagrams consisting of several morphisms at play in non-abelian Hodge theory. Starting with the lower square of the diagram, we show that the Riemann-Hilbert correspondence $RH : \Rep_G(\Sigma) \to \Loc_G(\Sigma)$ and the map $\nu_{RH} : \Sigma_{\dR} \to \Sigma_{B}$ are compatible with the Betti and de Rham integral functors. 

\begin{lemma} \label{le commutivity of lower diagram}
The Riemann-Hilbert correspondence induces a commutative diagram 
\begin{equation*}
\begin{tikzcd}[column sep = huge]
\Perf( \Sigma_{\dR} ) 
\arrow[r, "\Psi_{\dR, \rho}"] 
\arrow[d, "\nu_{RH,*}"']  & \Perf( \Loc_G(\Sigma) ) \arrow[d, "RH^{*}"] \\
\Perf( \Sigma_{B} ) \arrow[r, "\Psi_{B,\rho}"'] & \Perf( \Rep_G(\Sigma) ) 
\end{tikzcd}.
\end{equation*}
\end{lemma}
\begin{proof}
Consider the following commutative diagram with Cartesian squares, where each vertical and diagonal arrow is a natural projection.
\begin{equation*}
\begin{tikzcd} 
\Sigma_B \arrow[dr, phantom, "\square"] & \Sigma_{\dR} \arrow[l, "\nu_{RH}"'] & \\
\Sigma_B \times \Rep_G(\Sigma) \arrow[u, "b_1"] \arrow[dr, "b_2"']
& \Sigma_{\dR} \times \Rep_G(\Sigma) \arrow[l, "\nu_{RH} \times \id"] \arrow[u, "a_1"] \arrow[d, "a_2"'] \arrow[r, "\id \times RH"'] \arrow[dr, phantom, "\square"] 
& \Sigma_{\dR} \times \Loc_G(\Sigma) \arrow[ul, "d_1"'] \arrow[d, "d_2"] \\
& \Rep_G(\Sigma) \arrow[r, "RH"']
& \Loc_G(\Sigma)
\end{tikzcd}. 
\end{equation*}
We apply the base change equivalences $b_1^{*} \nu_{RH,*} \simeq (\nu_{RH} \times \id)_{*} a_1^{*}$ and $Ra_{2,*} (\id \times RH)^{*} \simeq RH^{*} Rd_{2,*}$ and the universal isomorphism $(\nu_{RH} \times \id)^{*} \Vv_B \cong (\id \times RH)^{*} \Vv_{\dR}$, alongside the projection formula and various functorialities. With a fixed object $\Ee \in \Perf(\Sigma_{\dR})$, we have 
\begin{align*}
\Psi_{B, \rho}(\nu_{RH, *} \Ee) 
& \cong Rb_{2,*} \Big( \Vv_B \otimes b_1^{*} \nu_{RH, *} \Ee \Big) \\
& \cong Rb_{2,*} \Big( \Vv_B \otimes (\nu_{RH} \times \id)_{ *} a_1^{*} \Ee \Big)
\\
& \cong Rb_{2,*} (\nu_{RH} \times \id)_{ *} \Big( (\nu_{RH} \times \id)^{*} \Vv_B \otimes a_1^{*} \Ee \Big) \\
& \cong Ra_{2,*} \Big( (\id \times RH)^{*} \Vv_{\dR} \otimes a_1^{*} \Ee \Big) \\
& \cong Ra_{2,*} \Big( (\id \times RH)^{*} \Vv_{\dR} \otimes (\id \times RH)^{*} d_1^{*} \Ee \Big) \\
& \cong Ra_{2,*} (\id \times RH)^{*} \Big( \Vv_{\dR} \otimes d_1^{*} \Ee \Big) \\
& \cong RH^{*} Rd_{2,*}\Big( \Vv_{\dR} \otimes d_1^{*} \Ee \Big). \qedhere
\end{align*}
\end{proof}

Next we show that the equivalences 
\[
\Triv_{\Sigma}: \Sigma_{\Hod}|_{\AA^1 - 0} \to \Sigma_{\dR} \times (\AA^1 - 0), \quad \Triv : \Hodge_G(\Sigma)|_{\AA^1 - 0} \to \Loc_G(\Sigma) \times (\AA^1 - 0),
\]
give rise to a compatibility between the Hodge and de Rham integral functors. 

\begin{lemma}  \label{le commutivity of upper diagram}
\label{le hodge and de rham functors}
The trivialisations $\Triv$ and $\Triv_{\Sigma}$ induce a commutative diagram
\begin{equation*}
\begin{tikzcd}[column sep = huge]
\Perf( \Sigma_{\Hod}|_{\AA^1 - 0} ) 
\arrow[r, "\Psi_{\Hod,\rho}|_{\AA^1 - 0}"]
\arrow[d, "\Triv_{\Sigma, *}" ]
& \Perf ( \Hodge_G(\Sigma)|_{\AA^1 - 0} ) \arrow[d, "\Triv_{*}" ]\\
\Perf ( \Sigma_{\dR} \times (\AA^1 - 0) ) \arrow[r, "\widehat \Psi_{\dR, \rho}"] 
& \Perf( \Loc_G(\Sigma) \times (\AA^1 - 0) ) .
\end{tikzcd}.
\end{equation*}
\end{lemma}

\begin{proof}
Consider the following Cartesian diagrams, where the horizontal arrows are natural projections. 
\begin{equation*}
\begin{tikzcd}[column sep = huge]
\Sigma_{\dR} \times (\AA^1 - 0) 
\arrow[dr, phantom, "\square"]
& 
\Sigma_{\dR} \times (\AA^1 - 0) \times \Loc_G(\Sigma) \times (\AA^1 - 0)
\arrow[l, "\pi_{12}"'] 
\arrow[r, "\pi_{34}"] 
\arrow[dr, phantom, "\square"]
& \Loc_G(\Sigma) \times (\AA^1 - 0)
\\
\Sigma_{\Hod}|_{\AA^1 - 0} \arrow[u, "\Triv_{\Sigma}"]
& 
\Sigma_{\Hod}|_{\AA^1 - 0} \times \Hodge_G(\Sigma)|_{\AA^1 - 0} \arrow[l, "h_1"] \arrow[r, "h_2"'] \arrow[u, "\Triv_{\Sigma} \times \Triv"]
& \Hodge_{G}(\Sigma) |_{\AA^1 - 0} \arrow[u, "\Triv"']
\end{tikzcd}.
\end{equation*}
We apply a universal isomorphism $\Vv_{\Hod} |_{\AA^1 - 0} \cong (\Triv_\Sigma \times \Triv)^{*} \widehat{\Vv}_{\dR}$ and a base change equivalence $(\Triv_{\Sigma} \times \Triv)_{*} h_1^{*} \simeq \pi_{12}^{*} \Triv_{\Sigma, *}$,  alongside the projection formula and various functorialities. With a fixed $\Ff \in \Perf( \Sigma_{\Hod}|_{\AA^1 - 0})$, we have the following isomorphisms: 
\begin{align*}
\Triv_{*} \circ \Psi_{\Hod,\rho}|_{\AA^1 - 0}(\Ff |_{\AA^1 - 0} ) 
& \cong \Triv_{*} \circ Rh_{2,*} \Big( \rho(\Vv_{\Hod}|_{\AA^1 - 0}) \otimes h_1^{*} \Ff \Big) \\
& \cong \pi_{34,*} ( \Triv_{\Sigma} \times \Triv )_{*} 
\Big( ( \Triv_{\Sigma} \times \Triv )^{*} \rho(\widehat{\Vv}_{\dR}) \otimes h_1^{*} \Ff \Big) \\
& \cong \pi_{34,*} 
\Big( \rho(\widehat{\Vv}_{\dR}) \otimes ( \Triv_{\Sigma} \times \Triv )_{*} h_1^{*} \Ff \Big) \\
& \cong \pi_{34,*} 
\Big( \rho(\widehat{\Vv}_{\dR}) \otimes \pi_{12}^{*} \Triv_{\Sigma,*} \Ff \Big) \\
& = \widehat{\Psi}_{\dR, \rho}( \Triv_{\Sigma,*} \Ff ). \qedhere
\end{align*}
\end{proof}

\section{Categories of (BBB)-branes}
\label{se BBB on Higgs}

Kapustin and Witten \cite{kapustin&witten} define (a non-exhaustive class of) (BBB)-branes to be hyperholomorphic submanifolds of $\Mm_{\Dol}(X,G)$ that support a hyperholomorphic bundle. Following a suggestion of Gaiotto \cite[Appendix C]{gaiotto}, this section proposes a categorification of (BBB)-branes in terms of perfect complexes on the Deligne-Hitchin twistor space\footnote{We thank Anna Sisak for recommending this point of view to the authors.}. Using the Deligne stack constructed in Definition \ref{de deligne stack}, we provide a natural generalisation of the categorified (BBB)-branes to the setting of derived algebraic and analytic stacks.  

\subsection{(BBB)-branes and twistor space}
\label{se twistoral BBB}

On a hyperk\"ahler manifold $(M, g, J_1, J_2, J_3)$, a \textit{hyperholomorphic bundle} $(E,\nabla)$ is a Hermitian and holomorphic bundle (i.e. curvature $F_{\nabla}$ is of type $(1, 1)$) with respect to all three complex structures. After removing the Hermitian condition, $(E,\nabla)$ is called an \textit{autodual bundle}. The twistor transform of Kaledin and Verbitsky \cite[§5]{kaledin&verbitsky} assigns a holomorphic bundle on $\Tw(M)$ to an autodual bundle on $M$, as we now describe. A twistor space $\Tw(M)$ always has analytic surjections $\tau : \Tw(M) \to \PP^1$ and $p: \Tw(M) \to M$ that commute with an analytic isomorphism $\Tw(M) \cong M \times \PP^1$. A twistor line is a section $\sigma: \PP^1 \to \Tw(M)$, which is called \textit{horizontal} if the composition $\PP^1 \xrightarrow{\sigma} \Tw(M) \xrightarrow{\cong} M \times \PP^1$ defines an isomorphism $\PP^1 \cong \{m\} \times M$, for some $m \in M$. The twistor transform of Kaledin-Verbitsky is an equivalence of categories 
\begin{equation}
\label{eq twistor transform}
\left\{ 
\begin{array}{l}
\text{autodual bundles}
\\
\qquad \text{on $M$} 
\end{array}
\right\}
\xrightarrow{\cong }
\left\{ 
\begin{array}{l}
\text{holomorphic bundles on $\Tw(M)$ that} \\
\text{are trivial on horizontal twistor lines}
\end{array}
\right\},
\end{equation}
\[
(B,\nabla) \longmapsto (p^{*}B , p^{*}\nabla^{0,1}),
\]
where $p^{*}\nabla^{0,1}$ is the Dolbeault operator $\ol\partial_{p^{*}B}$ for the pullback bundle. With $p^{*}B$ at hand, one retrieves $B$ by pullback along an embedding $\imath : M \hookrightarrow \Tw(M)$, for then one has the functorial isomorphism 
\begin{equation}
\label{eq inclusion pullback}
\iota^{*} (p^{*}B) = (p \circ \iota)^{*}B \cong B.
\end{equation}
On the hyperk\"ahler moduli space $M = \Mm_{\Dol}(\Sigma, G)$, we have Deligne and Simpson's explicit construction of the twistor space (Definition \ref{de twistor space}) in terms of $\lambda$-connections. One can therefore realise the projection $\tau : \Tw(\Mm_{\Dol}(\Sigma, G)) \to \PP^1$ as the map that records the $\lambda$-value of a $\lambda$-connection. The fiber over $\lambda = 0$ is then $\imath_{\Dol} : \Mm_{\Dol}(\Sigma, G) \hookrightarrow \Tw(\Mm_{\Dol}(\Sigma, G))$. 

Since the hyperk\"ahler structure on $\Mm_{\Dol}(\Sigma, G)$ arises from non-abelian Hodge theory, the horizontal twistor lines on $\Tw(\Mm_{\Dol}(\Sigma, G))$ have an explicit description. They are precisely the sections defined by contracting the non-abelian Hodge map \eqref{eq NAH} applied to a $G$-Higgs bundle $(P,\theta) \to \Sigma$, 
\begin{equation}
\label{eq horiz}
\sigma_{(P,\theta)} := \NAH( (P , \theta) \, , \, \cdot \, ) : \PP^1 \to \Tw(\Mm_{\Dol}(\Sigma, G)).
\end{equation}
With the twistor geometry in place, we now analyse Gaiotto's suggestion for a category of (BBB)-branes \cite[Appendix C]{gaiotto}. The original (BBB)-branes of Kapustin and Witten \cite{kapustin&witten} are hyperholomorphic submanifolds of $\Mm_{\Dol}(\Sigma,G)$ that support a hyperholomorphic bundle. Acting by the twistor transform, (BBB)-branes correspond to locally free sheaves on $\Tw(\Mm_{\Dol}(\Sigma,G))$ that are trivial on horizontal twistor lines. Gaiotto indicates that one should consider more general (BBB)-branes defined by perfect\footnote{Gaiotto's motivation for considering \textit{perfect} complexes comes from the behaviour of supersymmetric ground states and their excitations. More directly, this could be justified by noting that $\Perf(\cdot) \subset \QC(\cdot)$ is a well-behaved dg-subcategory.} complexes on $\Tw(\Mm_{\Dol}(\Sigma,G))$. We therefore consider 
\[
\Perf^{\, =}(\Tw(\Mm_{\Dol}(\Sigma, G))) \subset \Perf(\Tw(\Mm_{\Dol}(\Sigma, G))),
\]
the full subcategory of complexes that are trivial on horizontal twistor lines. We say that a complex $\TT$ is trivial on a horizontal twistor line $\sigma : \PP^1 \to \Tw(\Mm_{\Dol}(\Sigma,G)$ if $\sigma^{*}\TT$ is quasi-isomorphic to a complex of free sheaves. 

Additionally, should one wish to consider (BBB)-brane data that includes an algebraic sheaf, one can consider the natural embedding $\lambda_{\Dol} : \Mm_{\Dol}(\Sigma, G) \hookrightarrow \Mm_{\Dol}(X, G)$ and the corresponding analytification functor 
\[
( \cdot )^{an} := \lambda_{\Dol}^{*} : \Perf(  \Mm_{\Dol}(X , G) ) \to \Perf( \Mm_{\Dol}(\Sigma, G)  ). 
\]
We therefore propose a categorification of (BBB)-branes defined by the Cartesian square
\begin{equation}
\label{eq bbb category}
\begin{tikzcd}
\BBB( \Mm_{\Dol}(X, G) )  
\arrow[dr, phantom, "\square"] 
\arrow[r]
\arrow[d]
& \Perf^{\, =}(\Tw(\Mm_{\Dol}(\Sigma, G))) \arrow[d, "L\imath_{\Dol}^{*}"] \\
\Perf(  \Mm_{\Dol}(X , G) ) \arrow[r, "(\cdot)^{an}"'] & \Perf( \Mm_{\Dol}(\Sigma, G)  )
\end{tikzcd}.
\end{equation}
Objects are pairs $(\BB, \TT)$ consisting of $\BB \in \Perf(  \Mm_{\Dol}(X , G) )$ and $\TT \in \Perf^{\, =}(\Tw(\Mm_{\Dol}(\Sigma, G)))$ such that $\BB^{an} \cong \imath_{\Dol}^{*}\TT$ and morphisms are 
\[
\Hom\big((\BB_1, \TT_1), (\BB_2, \TT_2)\big) = \big\{ (b,t) \in \Hom(\BB_1, \BB_2) \times \Hom(\TT_1, \TT_2) 
\text{ such that $b^{an} \simeq L\imath_{\Dol}^{*} t$ } \big\}.
\]
Heuristically, $\TT $ encodes the hyperholomorphic geometry of $\BB$ via hyperk\"ahler rotations, achieved by application of non-abelian Hodge theory. By \textit{categorification of (BBB)-branes}, we refer to the following conclusive remark. 

\begin{proposition}
\label{co BBB category}
The (BBB)-branes on $\Mm_{\Dol}(X,G)$ defined by Kapustin-Witten correspond to locally free objects in $\BBB(\Mm_{\Dol}(X,G))$ that are supported in one homological degree.
\end{proposition}

A fundamental example of a (BBB)-brane is a \textit{zerobrane}: a $\CC$-valued point $(P,\theta) \in \Mm_{\Dol}(X,G)$ equipped with the length one skyscraper $\Oo_{(P,\theta)}$. In the category of (BBB)-branes, zerobranes correspond to objects $(\BB, \TT) = (\Oo_{(P,\theta)}, \Oo_{\im(\sigma_{(P,\theta)})})$, where $\TT$ is the skyscraper sheaf on the image of the twistor line of $(P,\theta)$.

\begin{remark}
\label{re BBB appeals}
The twistor framework has the following appealing properties: 

\begin{enumerate}
    \item The perfect condition can be removed by performing the same constructions with $\DCoh(\cdot)$.

    \item An object of $\BBB(\Mm_{\Dol}(X,G))$ maps to a standard B-brane on $\Mm_{\Dol}(X,G))$ in every complex structure. For $J_{\lambda}$, this is achieved by restricting $\TT$ to the fiber $\Tw(\Mm_{\Dol}(\Sigma, G)) \times_{\PP^1} \{\lambda\}$ and pulling back along the non-abelian Hodge map 
    \[
    \NAH( \, \cdot \, , \lambda) : \Mm_{\Dol}(X,G)) \to \Tw(\Mm_{\Dol}(\Sigma, G)) \times_{\PP^1} \{\lambda\}.
    \]
    
    \item Cohomology functors $H^k(\cdot) : \DCoh( \Mm_{\Dol}(X , G)) \to \Coh( \Mm_{\Dol}(X , G))$ preserve $\BBB( \Mm_{\Dol}(X , G) )$. This is because an isomorphism $\BB^{an} \cong \imath^{*}_{\Dol}\TT$ induces one between 
    \[
    H^k( \BB)^{an} \cong H^k( \BB^{an} ) \cong H^k( \imath^{*}_{\Dol}\TT ) \cong \imath^{*}_{\Dol}H^k( \TT ) .
    \]

    \item The same construction applies to any hyperk\"ahler manifold $M$ to define a category $\BBB(M)$, although it is not clear if the notion of (BBB)-brane is meaningful for physics or mirror symmetry in the general context.
\end{enumerate}
\end{remark}

\subsection{(BBB)-branes and the Deligne stack}
\label{se bbb}

We propose a notion of (BBB)-brane on $\Higgs_G(X)$ by considering $\Deligne_G(\Sigma)$ in the same role that $\Tw(\Mm_{\Dol}(\Sigma, G))$ plays in the previous section. Our approach suggests “$\Deligne_G(\Sigma) = \Tw(\Higgs_G(\Sigma))$", although we stress there does not yet exist a (shifted/derived) hyperk\"ahler or twistor geometry for stacks. Our strategy uses descent to the well-known twistor theory for the moduli spaces that arise in non-abelian Hodge theory. We view our proposal as an \textit{initial ansatz}, to be refined by a deeper analysis of mirror symmetry, geometric Langlands and future developments in hyperk\"ahler geometry.

Holstein-Porta's GAGA-type theory for derived mapping stacks \cite[Corollary 6.15]{holstein&porta} establishes the analytification
\[
\Higgs_G(X)^{an} \cong \Higgs_G(X^{an}) = \Higgs_G(\Sigma),
\]
so one can consider the exact analytification functor
\begin{equation}
\label{eq analytic functor higgs}
(\cdot)^{an} : \Perf(\Higgs_G(X)) \to \Perf(\Higgs_G(\Sigma)).
\end{equation}
Recall the construction of the Deligne stack comes with a natural morphism $\tau : \Deligne_G(\Sigma) \to \PP^1$ that takes the $\lambda$-value of a $\lambda$-connection (see \eqref{eq define tau}). As for the moduli spaces, we call $\sigma : \PP^1 \to \Deligne_G(\Sigma)$ a \textit{twistor line} if $\sigma$ is a section of $\tau$. To keep track of these sections we introduce the notation\footnote{Sections glue to define stacks $\Sect(\PP^1, \Deligne_G(\Sigma)) \subset \Maps(\PP^1, \Deligne_G(\Sigma))$, similar to those studied by Ginzburg-Rozenblyum \cite{ginzburg&rozenblyum}, hence we borrow their notation.}
\[
\Sect(\PP^1, \Deligne_G(\Sigma)) \subset \Hom(\PP^1, \Deligne_G(\Sigma)).
\]
The horizontal condition on twistor lines is problematic. Non-abelian Hodge theory does not apply to the stacks involved, so the horizontal twistor lines defined in \eqref{eq horiz} no longer exist. We shall make use of a strategy that restricts to stability loci and descends to the moduli spaces. Recall from Proposition \ref{pr deligne moduli} the good moduli space structure
\[
f_{\Del} : t_0(\Deligne_{G}(\Sigma))^{\sst} \to \Tw(\Mm_{\Dol}(\Sigma, G)),
\]
constructed as a gluing of Hodge moduli spaces over $\Sigma$ and $\ol\Sigma$. Every twistor line gives rise to a truncated morphism $t_0(\sigma) : \PP^1 \to t_0(\Deligne_G(\Sigma))$ of classical 1-stacks. If $t_0(\sigma)$ happens to land in the stable locus, we say that $\sigma$ is \textit{horizontal} if the composition 
\begin{equation}
\label{eq sst horizontal}
f_{\Del} \circ t_0(\sigma) : \PP^1 \to \Tw(\Mm_{\Dol}(\Sigma, G)).
\end{equation}
is a horizontal twistor line in the sense of twistor theory (\textit{i.e.} of the form \eqref{eq horiz}). We also consider unstable $\lambda$-connections that become stable upon reduction to a subgroup of $G$. In this way, we propose the following more general notion of horizontal\footnote{We thank Eric Chen for addressing the need for horizontal twistor lines mapping to the unstable locus.}.  
\begin{definition}
\label{de stacky horizontal twistor lines}
A twistor line $\sigma : \PP^1 \to \Deligne_{G}(\Sigma)$ is \textit{horizontal} if there exist a reductive subgroup $G' \hookrightarrow G$ and a map $\sigma' : \PP^1 \to t_0(\Deligne_{G'}(\Sigma))^{\sst}$ such that $t_0(\sigma)$ factors as 
\[
\begin{tikzcd}
& t_0(\Deligne_{G}(\Sigma))  \\
\PP^1 \arrow[ur, "t_0(\sigma)"] 
\arrow[r, "\sigma'"']
& t_0(\Deligne_{G'}(\Sigma))^{\sst} 
\arrow[r, "f_{\Del}"']
\arrow[u, "\xi_{\Del}"']
& \Tw(\Mm_{\Dol}(\Sigma, G'))
\end{tikzcd},
\]
where $f_{\Del} \circ \sigma'$ is a horizontal twistor line on $\Tw(\Mm_{\Dol}(X,G'))$ and $\xi_{\Del}$ is extension of structure group from $G'$ to $G$ (\textit{i.e.} pushforward along $BG' \to BG$).
\end{definition}

We keep track of horizontal twistor lines with the notation $\Sect^{\, =}(\PP^1, \Deligne_{G}(\Sigma))$. Also let $\Perf^{\,=}(\Deligne_G(\Sigma)) \subset \Perf(\Deligne_G(\Sigma))$ denote the full triangulated subcategory consisting of perfect complexes that are trivial on horizontal twistor lines. As for complexes on the twistor space, we say that $\TT$ is trivial on $\sigma : \PP^1 \to \Deligne_G(\Sigma)$ when $\sigma^{*}\TT$ is quasi-isomorphic to a complex of free sheaves. 

The following result is a simple test for when $\TT \in \Perf(\Deligne_G(\Sigma))$ is horizontal in terms of descent. 

\begin{lemma}
\label{lm compatibility of trivializations}
Consider $\TT \in \Perf(\Deligne_G(\Sigma))$ such that for every subgroup $G' \hookrightarrow G$, there exists $\TT' \in \Perf^{\,=}(\Tw(\Mm_{\Dol}(\Sigma, G'))$ and an isomorphism 
\[
\xi_{\Del}^{*}\TT \cong f_{\Del}^*\TT'.
\]
Under these conditions, $\TT$ is an object of $\Perf^{\,=}(\Deligne_{G}(\Sigma))$.   
\end{lemma} 

\begin{proof}
A horizontal twistor line $\sigma \in \Sect^{\, =} (\PP^1 , \Deligne_G(\Sigma))$ factors through the classical truncation $t_0(\sigma) : \PP^1 \to t_0(\Deligne_G(\Sigma))$ and the embedding $t_0(\Deligne_G(\Sigma)) \to \Deligne_G(\Sigma)$. By the definition of horizontal, there exists a $\sigma'$ such that $t_0(\sigma) = \xi_{\Del} \circ \sigma'$. By functoriality and the hypothesis of the lemma, one has 
\[
\sigma^{*}\TT \cong t_0(\sigma)^*t_0(\TT) \cong (\sigma')^* \xi_{\Del}^* \TT \cong (\sigma')^* f_{\Del}^* \TT' \cong (f_{\Del} \circ \sigma')^*\TT'.
\]
$f_{\Del} \circ \sigma'$ is horizontal on $\Tw(\Mm_{\Dol}(\Sigma, G'))$, so therefore $\TT'$ trivialises on $f_{\Del} \circ \sigma'$. It follows that $\TT$ trivialises on $\sigma$. 
\end{proof}


Having established which twistor lines are horizontal, we propose a corresponding notion of (BBB)-brane on $\Higgs_G(X)$. 

\begin{definition}
\label{de stacky BBB}
\label{de stacky BBB complex}
%
%
Let $X$ be a smooth projective curve with analytification $\Sigma = X^{an}$. Let us define the category $\BBB( \Higgs_G(X) )$ via the Cartesian square
\begin{equation}
\label{eq bbb category for stacks}
\begin{tikzcd}
\BBB( \Higgs_G(X) )  
\arrow[dr, phantom, "\square"] 
\arrow[r]
\arrow[d]
& \Perf^{\, =}(\Deligne_G(\Sigma)) \arrow[d, "L\imath_{\Higgs}^{*}"] \\
\Perf(  \Higgs_G(X) ) \arrow[r, "(\cdot)^{an}"'] & \Perf( \Higgs_G(\Sigma)  )
\end{tikzcd}.
\end{equation}
Objects are pairs $(\BB, \TT)$ consisting of $\BB \in \Perf(\Higgs_G(X))$ and $\TT \in \Perf^{\, =}(\Deligne_G(\Sigma))$ such that $\BB^{an} \cong \imath_{\Higgs}^{*}\TT$ and morphisms are 
\[
\Hom\big((\BB_1, \TT_1), (\BB_2, \TT_2)\big) = \big\{ (b,t) \in \Hom(\BB_1, \BB_2) \times \Hom(\TT_1, \TT_2) 
\text{ such that $b^{an} \simeq L\imath_{\Higgs}^{*} t$ } \big\}.
\]
\end{definition}

\begin{remark}
A more restrictive definition of $\Perf^{\,=}(\Deligne_G(\Sigma)) \subset \Perf(\Deligne_G(\Sigma))$ (hence of $\BBB(\Higgs_G(X))$) could be made by considering a wider class of horizontal twistor lines. We expect suitable conditions to arise from a deeper analysis of mirror symmetry and geometric Langlands. 
\end{remark} 

The same diagram \eqref{eq bbb category for stacks} with the semistable loci defines the category $\BBB(\Higgs_G(X)^{\sst})$ of (BBB)-branes supported on $t_0(\Higgs_G(X))^{\sst} \subset \Higgs_G(X)$. This is achieved with the subcategory $\Perf^{\,=}(t_0(\Deligne_G(\Sigma))^{\sst}) \subset \Perf(\Deligne_G(\Sigma)^{\sst})$ of complexes supported on the classical and semistable locus. In the notation of Definition \ref{de stacky horizontal twistor lines}, the notion of horizontal twistor lines in this case is achieved by the discussion surrounding \eqref{eq sst horizontal} (\textit{i.e.} by the choice $G' = G$ in Definition \ref{de stacky horizontal twistor lines}). We place these categories into the following cube, in which the front and back faces are the Cartesian squares that define the respective (BBB)-brane categories. 

\begin{equation*}
\begin{tikzcd}[row sep={40,between origins}, column sep={70,between origins}]
& \BBB(\Mm_{\Dol}(X,G)) \ar{rr}\ar{dd} 
\arrow[dl, dashed, "(f_{\Dol} \times f_{\Del})^{*}"'] 
& & \Perf^{\, =}(\Tw(\Mm_{\Dol}(\Sigma, G))) \vphantom{c} \ar{dd} \ar[dl, "f_{\Del}^{*}"]
\\
\BBB(t_0(\Higgs_G(X))^{\sst}) \ar[crossing over]{rr} \ar{dd} & & \Perf^{\, =}(t_0(\Deligne_G(\Sigma))^{\sst})
\\
& \Perf(\Mm_{\Dol}(X,G)) \ar{rr} \ar[dl, "f_{\Dol}^{*}"'] & & \Perf(\Mm_{\Dol}(\Sigma, G)) \vphantom{f} \ar[dl, "f_{\Dol}^{an, *}"] 
\\
\Perf(t_0(\Higgs_G(X))^{\sst}) \ar{rr} && \Perf(t_0(\Higgs_G(\Sigma))^{\sst}) \ar[from=uu,crossing over] 
\end{tikzcd}
\end{equation*}

The following fills in the dashed arrow above, showing that (BBB)-branes on the stack naturally contains (BBB)-branes that are supported on the moduli space. 

\begin{proposition}
Consider a (BBB)-brane $(\BB, \TT) \in \BBB(\Mm_{\Dol}(X,G))$. Then, its pullback under the structural morphisms of good moduli spaces 
\[
f_{\Dol} : \Higgs_{G}(X)^{\sst} \to  \Mm_{\Dol}(X,G), \quad \text{and} \quad f_{\Del} : \Deligne_{G}(\Sigma)^{\sst} \to  \Tw(\Mm_{\Dol}(\Sigma,G))
\]
produces an element of $\BBB(t_0(\Higgs_G(X))^{\sst})$.
\end{proposition} 

\begin{proof} 
The proof follows from some functorial observations. Firstly, the pullback of a perfect complex is again perfect, so both $f_{\Dol}^{*}\BB$ and $f_{\Del}^{*}\TT$ are perfect. Secondly,  
Lemma \ref{lm compatibility of trivializations} confirms that $f_{\Del}^*$ preserves the subcategories $\Perf^{\, = }(\cdot)$, so 
\[
f_{\Del}^{*}\TT \in \Perf^{\, = }(t_0(\Higgs_G(\Sigma))^{\sst}).
\]
Thirdly, we check the pair $(f_{\Dol}^{*}\BB, f_{\Del}^{*}\TT)$ is compatible with restriction along $\Higgs_G(\Sigma) \hookrightarrow \Deligne_G(\Sigma)$. Observe that one can consider the good moduli spaces sitting in the commutative diagram
\[
\begin{tikzcd}
t_0(\Higgs_{G}(\Sigma))^{\sst}   
\arrow[dr, phantom, "\square"] 
\arrow[r, "t_0(\imath_{\Higgs})"]
\arrow[d, "f_{\Dol}^{an}"']
& t_0(\Deligne_{G}(\Sigma))^{\sst} \arrow[d, "f_{\Del}"] \\ \Mm_{\Dol}(\Sigma, G)   
\arrow[r, "\imath_{\Dol}"']
& \Tw(\Mm_{\Dol}(\Sigma, G))
\end{tikzcd}, 
\]
where $t_0(\imath_{\Higgs})$ is the restriction of $\imath_{\Higgs}$ to the classical locus. Then,  
\[
Lt_0(\imath_{\Higgs})^{*} f_{\Del}^* \TT \cong (f_{\Dol}^{an})^* L\imath_{\Dol}^{*} \TT \cong (f_{\Dol}^{an})^* \BB^{an} \cong (f_{\Dol}^* \BB)^{an},
\]
where the last isomorphism follows from the compatibility between pullback and analytification.
\end{proof}

\section{Examples of (BBB)-branes}
\label{se nahm branes}

With respect to our new notion of (BBB)-brane, we now place such a structure on the Dirac-Higgs complex $\DD_{\rho} = Rp_{2,*}(\rho(\Uu)) \in \Perf(\Higgs_G(X))$. More generally, we construct a family of (BBB)-branes with underlying complex
\begin{equation}
\label{eq dol branes intro}
\Phi_{\Bon,\rho}(E,\phi) = Rp_{2,*}(\rho(\Uu) \otimes (E,\phi)).
\end{equation}
We call these \textit{integral branes}. On the $GL_n$ moduli spaces, the Dirac-Higgs bundle has previously been understood as a (BBB)-brane \cite{hitchin_dirac}, \cite[Theorem 2.6.3]{blaavand}, and the twisted family \eqref{eq dol branes intro} resembles branes studied by Bonsdorff, Franco-Jardim and Frejlich-Jardim \cite{bonsdorff_1, bonsdorff_2, franco&jardim, frejlich&jardim}. Our constructions shall fill in the dashed arrow in the diagram 
\begin{equation}
\adjustbox{scale=0.93,center}{%
\label{eq construct BBB}
\begin{tikzcd}
\PHgs(\Sigma, K_{\Sigma}) \times_{\PHgs(X, K_X)} (\Perf(\Sigma_{\Hod}) \times_{B} \Perf(\ol{\Sigma}_{\Hod}))^{\, =} \arrow[ddr, bend right=20, "\Phi_{\Bon, \rho}"'] \arrow[dr, dashed] \arrow[drr, bend left=20, "\Psi_{\Del, \rho}"] &  & \\
& \BBB( \Higgs_G(X) )  
\arrow[dr, phantom, "\square"] 
\arrow[r]
\arrow[d]
& \Perf^{\, =}(\Deligne_G(\Sigma)) \arrow[d, "L\imath_{\Dol}^{*}"] 
\\
& \Perf(  \Higgs_G(X))  \arrow[r, "(\cdot)^{an}"'] & \Perf( \Higgs_G(\Sigma)  )
\end{tikzcd}}
\end{equation}
A suitable pair $(\BB, \TT) = ( \Phi_{\Bon, \rho}(E,\phi) , \Psi_{\Del, \rho}( \widetilde{\Ff} ) )$ becomes a (BBB)-brane after we compute the analytification $\BB^{an}$ and check that $\TT$ pulls back to $\BB^{an}$. The following section is dedicated to verifying these conditions.

\subsection{Computing analytifications}
\label{se GAGA}
Recall the embeddings 
\[
\lambda_X : \Sigma = X^{an} \hookrightarrow X,
\]
\[
\lambda_{\Higgs} : \Higgs_G(\Sigma) \cong \Higgs_G(X)^{an} \hookrightarrow \Higgs_G(X),
\]
that pullback to define the analytification functors 
\[
(\cdot)^{an} : \Perf(X) \to \Perf(\Sigma),
\]
\[
(\cdot)^{an} = \lambda_{\Higgs}^{*} : \Perf(\Higgs_G(X)) \to \Perf(\Higgs_G(\Sigma)).
\]
Recall the algebraic and analytic Bonsdorff functors (Definition \ref{de bonsdorff}) are defined to be the integral functors with integral kernels $\Uu$ and $\Uu^{an}$, 
\[
\Phi_{\Bon, \rho} : \PHgs(X,K_X) \to \Perf( \Higgs_G(X) ), \quad
(E,\phi) \mapsto Rp_{2,\star} \Big( \Vv_{\Dol} \otimes p_1^{*}(E,\phi)  \Big),
\]
\[
\Psi_{\Bon, \rho} : \PHgs(\Sigma, K_\Sigma)) \to \Perf( \Higgs_G(\Sigma) ), \quad 
(E,\phi) \mapsto R(p_2^{an})_{\star} \Big( \Vv^{an}_{\Dol} \otimes p^{an,*}_1(E,\phi)  \Big), 
\]
where $p_{2,\star}$ and $(p_2^{an})_{\star}$ are given by integrating in Dolbeault cohomology. The following confirms their relation via analytification. 

\begin{proposition}
\label{pr analytification}
Fix $(E,\phi) \in \Hgs(X, K_X)$ and take the analytification $(E,\phi)^{an} := \lambda_X^{*}(E,\phi) \in \Hgs(\Sigma, K_\Sigma)$. Then, in $\Perf(\Higgs_G(\Sigma))$, there exists an isomorphism 
\[
\Psi_{\Bon, \rho}( (E,\phi)^{an} ) \cong (\Phi_{\Bon, \rho} (E,\phi) ) ^{an}.
\]
\end{proposition}
\begin{proof}
Fix $\Ff \in \Perf(X \times \Higgs_G(X))$. Since the projection $p_2 : X \times \Higgs_G(X) \to \Higgs_G(X)$ is proper, we can apply a GAGA-type theorem of Porta-Yu \cite[Theorem 1.2]{porta&yu}, which provides an isomorphism
\begin{equation} \label{eq push-forward and analytification}
(Rp_{2,*} \Ff)^{an} \cong R(p_{2}^{an})_{*} \Ff^{an}.
\end{equation}
Observe as well that $Rp_{2,\star} \Big( \rho(\Uu) \otimes p_1^{*}(E,\phi) \Big)$ can be obtained as the cone of the map 
\[
Rp_{2,*} \Big( p_1^{*}(E,\phi) \Big) \to Rp_{2,*} \Big( \rho(\Uu) \otimes p_1^{*}(E,\phi) \Big),
\]
and similarly, $R(p_2^{an})_{\star} \Big( \rho(\Uu^{an}) \otimes (p_1^{*}(E,\phi))^{an} \Big)$ is the cone of the corresponding analytic morphism. In combination with \eqref{eq push-forward and analytification} this gives
\[
\Big(Rp_{2,\star} \Big( \rho(\Uu) \otimes p_1^{*}(E,\phi) \Big) \Big)^{an} \cong R(p_2^{an})_{\star} \Big( \rho(\Uu^{an}) \otimes (p_1^{*}(E,\phi))^{an} \Big).
\]
From the commutative square 
\begin{equation*}
\begin{tikzcd}[column sep = huge]
\Sigma \times \Higgs_G(\Sigma) 
\arrow[r, "\lambda_{X} \times \lambda_{\Higgs}"]
\arrow[d, "p_1^{an}"'] 
& X \times \Higgs_G(X) 
\arrow[d, "p_1"] \\
\Sigma \arrow[r, "\lambda_{X}"'] & X
\end{tikzcd},
\end{equation*}
it follows that, for any Higgs sheaf $(E,\phi) \in \Hgs(X,K_X)$, 
\[
(p_1^{*}(E,\phi))^{an} 
= (\lambda_{X} \times \lambda_{\Higgs})^{*} \circ p_1^{*}(E,\phi) 
\cong p_1^{an,*} \lambda_{X}^{*} (E,\phi) = p_1^{an,*} (E,\phi)^{an}.
\]
Combining the above isomorphisms provides 
\begin{align*}
(\Phi_{\Bon,\rho}(E,\phi))^{an} = Rp_{2,\star} \Big( \rho(\Uu) \otimes p_1^{*}(E,\phi) \Big)^{an} 
& \cong R(p_2^{an})_{\star} \Big( \rho(\Uu^{an}) \otimes (p_1^{*}(E,\phi))^{an} \Big) \\
& \cong R(p_2^{an})_{\star} \Big( \rho(\Uu^{an}) \otimes p_1^{an,*} (E,\phi)^{an} \Big) \\
& = \Psi_{\Bon, \rho}( (E,\phi)^{an} ). \qedhere
\end{align*}
\end{proof}

\subsection{Pullback of Deligne functor}
\label{se pullback deligne}
Recall the embeddings 
\begin{equation}
\label{eq embeddings}
\begin{tikzcd}[column sep = huge]
\Higgs_G(\Sigma) \arrow[dr, hookrightarrow, "\jmath_{\Higgs}"'] \arrow[r,  hookrightarrow, "\imath_{\Higgs}"]   & \Deligne_G(\Sigma) \\
 & \Hodge_G(\Sigma) \arrow[u, hookrightarrow, "\imath_{\Hodge}"]
\end{tikzcd}, 
\quad 
\begin{tikzcd}[column sep = huge]
\Sigma_{\Dol}
\arrow[r,  hookrightarrow, "\jmath_{\Dol}"]   
& \Sigma_{\Hod} 
\end{tikzcd},
\end{equation}
and consider the natural functor 
\[
\Perf(\Sigma_{\Hod}) \times_B \Perf(\ol{\Sigma}_{\Hod}) \to \Perf(\Sigma_{\Dol}), \quad (\Ff, \ol \Ff, f) \mapsto L\jmath_{\Dol}^{*} \Ff. 
\]
\begin{proposition}
\label{pr dolbeault restriction}
The Deligne and Dolbeault functors fit into a commutative diagram  
\begin{equation*}
\begin{tikzcd}[column sep = huge]
\Perf(\Sigma_{\Hod}) \times_B \Perf(\ol{\Sigma}_{\Hod}) \arrow[d, "L\jmath_{\Dol}^{*}"'] \arrow[r, "\Psi_{\Del,\rho}"]  & \Perf( \Deligne_G(\Sigma) ) \arrow[d, "L\imath_{\Higgs}^{*}"] \\
\Perf( \Sigma_{\Dol} )  \arrow[r, "\Psi_{\Dol,\rho}"'] & \Perf( \Higgs_G(\Sigma) ) 
\end{tikzcd}.
\end{equation*}
\end{proposition}

\begin{proof}
The isomorphism $\Sigma_{\Dol} \times \Higgs_G(\Sigma) \cong \Sigma_{\Hod} \times_{\AA^1} \Higgs_G(\Sigma)$ provides a Cartesian square 
\begin{equation*}
\begin{tikzcd} 
\Sigma_{\Dol} 
\arrow[r, "\jmath_{\Dol}"]
& \Sigma_{\Hod} \\
\Sigma_{\Dol} \times \Higgs_G(\Sigma) 
\arrow[u, "p_1"]
\arrow[r, "\id \times \jmath_{\Higgs}"] 
\arrow[d, "p_2"'] 
\arrow[dr, phantom, "\square"] 
& \Sigma_{\Hod} \times_{\AA^1} \Hodge_G(\Sigma) \arrow[d, "h_2"] \arrow[u, "h_1"'] \\
\Higgs_G(\Sigma) \arrow[r, "\jmath_{\Higgs}"'] & \Hodge_G(\Sigma)
\end{tikzcd},
\end{equation*}
and base change provides the equivalence $L\jmath_{\Higgs}^{*} \circ Rh_{2,*} \simeq Rp_{2,*} \circ L(\id \times \jmath_{\Higgs})^{*}$. Combined with commutivity of the embeddings in \eqref{eq embeddings}, commutivity of the upper square and the isomorphism $L(\id \times \jmath_{\Higgs})^{*} (\Vv_{\Hod})^{an} \cong (\Vv_{\Dol})^{an}$, we have 
\begin{align*}
L\imath_{\Higgs}^{*} \Psi_{\Del, \rho}(\Ff, \ol{\Ff}, f) 
& \cong L\jmath_{\Higgs}^{*} \Psi_{\Hod, \rho}(\Ff) \\
& = L\jmath_{\Higgs}^{*} Rh_{2,*} \Big( (\Vv_{\Hod})^{an} \otimes h_1^{*} \Ff \Big) \\
& \cong Rp_{2,*} \Big( L(\id \times \jmath_{\Higgs})^{*} (\Vv_{\Hod})^{an} \otimes L(\id \times \jmath_{\Higgs})^{*} h_1^{*} \Ff \Big) \\
& \cong Rp_{2,*} \Big((\Vv_{\Dol})^{an} \otimes p_1^{*} \circ L\jmath_{\Dol}^{*} \Ff \Big) \\
& = \Psi_{\Dol, \rho}\Big(L\jmath_{\Dol}^{*} \Ff \Big),
\end{align*}
which concludes the proof
\end{proof}

\subsection{Horizontal twistor lines}
\label{se twistor lines}
To construct complexes on the Deligne stack that are trivial on horizontal twistor lines, one can formally consider the base change 
\begin{equation}
\label{eq hodge =}
\begin{tikzcd}
\left( \Perf(\Sigma_{\Hod}) \times_B \Perf(\ol\Sigma_{\Hod}) \right)^{\,=} 
\arrow[r, "\Psi_{\Del, \rho}"]
\arrow[d, hook]
& \Perf^{\, =}(\Deligne_G(\Sigma)) 
\arrow[d, hook]
\\
\Perf(\Sigma_{\Hod}) \times_B \Perf(\ol\Sigma_{\Hod})
\arrow[r, "\Psi_{\Del, \rho}"] 
& \Perf(\Deligne_G(\Sigma))
\end{tikzcd}.
\end{equation}
A natural family of explicit examples are given as follows.
\begin{proposition}
\label{pr twistor lines} 
Fix an object $(\Ff, \ol{\Ff}, f) \in \Perf(\Sigma_{\Hod}) \times_B \Perf(\ol{\Sigma}_{\Hod})$ such that $f : \Ff_B \xrightarrow{\cong} \ol{\Ff}_B$ is equivalent to the identity map in $\Perf(\Sigma_B \times (\AA^1 - 0))$. Then, given a horizontal twistor line $\sigma \in \Sect^{\, =}( \PP^1, \Deligne_G(\Sigma))$, the pullback $\sigma^{*} \Psi_{\Del, \rho}(\Ff, \ol{\Ff}, f)$ is trivialisable. 

In other words, such an $(\Ff, \ol{\Ff}, f)$ is an object of $\left(\Perf(\Sigma_{\Hod}) \times_B \Perf(\ol{\Sigma}_{\Hod}) \right )^{\, =}$. 
\end{proposition}

\begin{proof}
Recall the Deligne functor transforms $f : \Ff_B \xrightarrow{\cong} \ol\Ff_B$ into an isomorpism 
\begin{equation}
\label{eq glue data}
g = \widehat{\Psi}_{B,\rho}(f) : (\Psi_{\Hod,\rho}(\Ff))_{\Rep} \xrightarrow{\cong} (\Psi_{\ol\Hod,\rho}(\ol\Ff))_{\Rep}.
\end{equation}
Let $\sigma_{\Rep} : \PP^1 - \{ 0,\infty \} \to \Rep_G(\Sigma) \times (\PP^1 - \{0,\infty\})$ be the restriction of the twistor line to the substack 
\[
\Rep_G(\Sigma) \times (\PP^1 - \{0,\infty\}) \subset \Deligne_G(\Sigma).
\] 
The pullback of the data in \eqref{eq glue data} along $\sigma_{\Rep}$ produces a pair of complexes on either hemispheres of $\PP^1 = \CC \cup \CC$ and a transition function over the equator $\PP^1 - \{0,\infty\}$. The transition function can be written 
\[
\sigma_{\Rep}^{*} \widehat{\Psi}_{B,\rho}(f)  :  \sigma_{\Rep}^{*} \Psi_{\Hod,\rho}(\Ff))_{\Rep} \to \sigma_{\Rep}^{*} \Psi_{\Hod,\rho}(\ol\Ff))_{\Rep}. 
\]
The hypothesis $f \cong \id$ transforms to an equivalence $\sigma_{\Rep}^{*} \widehat{\Psi}_{B,\rho}(f) \cong \id$, so one glues two complexes of free sheaves along the identity map. The result is that $\sigma^{*} \Psi_{\Del, \rho}(\Ff, \ol{\Ff}, f)$ is a complex of free sheaves.  
\end{proof}

\subsection{Construction of (BBB)-branes}

As indicated by the diagram \eqref{eq construct BBB}, our initial data for constructing (BBB)-branes consists of a vector
\[
(E,\phi, \Ff, \ol\Ff, f),
\] 
such that 
\begin{itemize}
    \item $(E,\phi) \in \PHgs(X, K_X)$ is an algebraic complex of locally free Higgs sheaves on $X$, 
    
    \item $(\Ff,\ol\Ff, f) \in \Perf(\Sigma_{\Hod}) \times_B \Perf(\ol\Sigma_{\Hod})$ is a product of perfect analytic complexes on Hodge shapes (see \eqref{eq Dom}), 

    \item Moreover $(\Ff,\ol\Ff, f) \in \left( \Perf(\Sigma_{\Hod}) \times_B \Perf(\ol\Sigma_{\Hod}) \right)^{\,=}$, where the decoration $(\cdot)^{^{\,=}}$ is defined by the Cartesian square \eqref{eq hodge =},  
    
    \item $\Ff$ restricts to $(E,\phi)$, or more precisely:
     \[
    \kappa^{\heart}(L\jmath_{\Dol}^{*}\Ff) \cong (E,\phi)^{an}.
    \]
\end{itemize}
We now check that evaluating the Bonsdorff and Deligne functors on this vector produces a family of (BBB)-branes.  

\label{se BBB theorems}
\begin{theorem}
\label{th BBB}
Given the data just described, the pair
\[
\BB := \Phi_{\Bon,\rho}(E,\phi) \in \Perf(\Higgs_G(X)),
\]
\[
\TT := \Psi_{\Del, \rho}(\Ff, \ol\Ff, f) \in \Perf^{\, =}(\Deligne_G(\Sigma)),
\]
defines an object $(\BB, \TT) \in \BBB(\Higgs_G(X))$. We call such objects \textbf{integral branes.}
\end{theorem}

\begin{proof}
It suffices to check that  $\BB^{an} \cong \iota_{\Higgs}^{*}\TT$. We verify this by moving objects around the commutative diagram 
\begin{equation}
\label{eq BBB structure}
\begin{tikzcd}[column sep = huge]
& \Perf(\Sigma_{\Hod}) \times_B \Perf(\ol{\Sigma}_{\Hod})
\arrow[r, "\Psi_{\Del, \rho}"] 
\arrow[d] 
& 
\Perf(\Deligne_G(\Sigma))
\arrow[d, "L\imath_{\Higgs}^{*}"]
\\
\Coh(\Sigma_{\Dol}) \arrow[d, "\kappa^{\heart}"'] \arrow[r, hookrightarrow] 
& \Perf(\Sigma_{\Dol})
\arrow[r, "\Psi_{\Dol, \rho}"] \arrow[d, "\kappa^{\heart}"']
&
\Perf( \Higgs_G(\Sigma) )
\arrow[d, equal] 
\\
\Hgs(\Sigma, K_\Sigma) \arrow[r, hookrightarrow] & \PHgs(\Sigma, K_\Sigma)
\arrow[r, "\Psi_{\Bon, \rho}"] 
& 
\Perf( \Higgs_G(\Sigma) )
\\
\Hgs(X, K_X) \arrow[r, hookrightarrow] \arrow[u, "(\cdot)^{an}"]& \PHgs(X, K_X)
\arrow[u, "(\cdot)^{an}"]
\arrow[r, "\Phi_{\Bon, \rho}"]
&
\Perf(\Higgs_G(X))
\arrow[u, "(\cdot)^{an}"']
\end{tikzcd}.
\end{equation}
Recall the upper square commutes by Proposition \ref{pr dolbeault restriction}, the middle square by Proposition \ref{pr Dol Mod commute} and the lower square by Proposition \ref{pr analytification}. The upper square computes the pullback
\[
L\imath^{*}_{\Higgs} \TT 
= L\imath^{*}_{\Higgs} \Psi_{\Del,\rho}(\Ff, \ol\Ff, f)
\cong \Psi_{\Dol,\rho} ( L\jmath_{\Dol}^{*}\Ff ).
\]
The hypothesis $\kappa^{\heart}(L\jmath_{\Dol}^{*}\Ff) \cong (E,\phi)^{an}$ alongside the commutivity of the middle square  provides  
\[
\Psi_{\Dol,\rho} ( L\jmath_{\Dol}^{*}\Ff_{\eta} ) 
\cong \Psi_{\Bon,\rho} \circ \kappa^{\heart}(L\jmath_{\Dol}^{*}\Ff_{\eta} )
\cong \Psi_{\Bon,\rho} ( (E,\phi)^{an} ).
\]
Combining the above formulae with the lower square yields
\[
L\imath^{*}_{\Higgs} \TT 
\cong \Psi_{\Bon,\rho} ( (E,\phi)^{an} ) 
\cong  \Phi_{\Bon,\rho} ( E,\phi ) ^{an} = \BB^{an}.
\]
The pair $(\BB,\TT) \in \Perf(\Higgs_G(X)) \times \Perf^{\, =}(\Deligne_G(\Sigma))$ therefore satisfies the conditions of Definition \ref{de stacky BBB} that define a (BBB)-brane. 
\end{proof}


We place the main example of this article into the context of our general existence theorem. 

\begin{corollary}
\label{co Dirac-Higgs is BBB}
In the notation of Theorem \ref{th BBB}, consider the vector 
\[
(\Oo_X, 0, \Oo_{\Sigma_{\Hod}}, \Oo_{\ol\Sigma_{\Hod}}, \id),
\]
where $\id$ is the identity map $\Oo_{\Sigma_B} \to \Oo_{\ol{\Sigma}_B} = \Oo_{\Sigma_B}$. We claim this vector transforms to a (BBB)-brane 
\[
\DD_{\rho} = \Phi_{\Bon,\rho}(\Oo_{X}, 0), \quad \TT_{\rho} = \Psi_{\Del, \rho}(\Oo_{\Sigma_{\Hod}}, \Oo_{\ol\Sigma_{\Hod}}, \id),
\]
with underlying complex given by the Dirac-Higgs complex $\DD_{\rho}$. 
\end{corollary}

\begin{proof}
It suffices to check that the given vector satisfies the conditions needed for Theorem \ref{th BBB}. In other words, that the vector is an object of the product 
\[
\Hgs(X, K_X) \times_{\Hgs(\Sigma, K_{\Sigma})} \left(\Perf(\Sigma_{\Hod}) \times_B \Perf(\ol\Sigma_{\Hod})\right)^{\, =}
\]
In the notation of Proposition \ref{pr satisfy gluing}, the induced sheaves on the Betti shapes are 
\[
(\Oo_{\Sigma_{\Hod}})_B = \Oo_{\Sigma_B \times (\AA^1 - 0)}, \quad (\Oo_{\ol{\Sigma}_{\Hod}})_B = \Oo_{\ol{\Sigma}_B \times (\AA^1 - 0)}.
\]
which, after applying $\Sigma_B = \ol{\Sigma}_B$, admits the choice of isomorphism 
\[
\id: (\Oo_{\Sigma_{\Hod}})_B \to (\Oo_{\ol{\Sigma}_{\Hod}})_B.
\]
The triple $(\Oo_{\Sigma_{\Hod}}, \Oo_{\ol\Sigma_{\Hod}}, \id)$ is therefore a well-defined object of $\Perf(\Sigma_{\Hod}) \times_B \Perf(\ol\Sigma_{\Hod})$. Moreover, it is an object of $\left(\Perf(\Sigma_{\Hod}) \times_B \Perf(\ol\Sigma_{\Hod})\right)^{\, =}$ by Proposition \ref{pr twistor lines}. 

It remains to confirm the compatibility between $(\Oo_X, 0)$ and $(\Oo_{\Sigma_{\Hod}}, \Oo_{\ol\Sigma_{\Hod}}, \id)$ with respect to the restriction $\imath_{\Higgs} : \Higgs_G(\Sigma) \hookrightarrow \Deligne_G(\Sigma)$. This follows from the observations $L\imath_{\Higgs}^{*} \Oo_{\Sigma_{\Hod}} \cong \Oo_{\Sigma_{\Dol}}$ and 
\[
\kappa^{\heart} (\Oo_{\Sigma_{\Dol}}) = ( \Oo_{\Sigma} , 0) = ( \Oo_X, 0)^{an}. \qedhere
\]
\end{proof}

\section{Relation with Wilson functors}
\label{se relation}

This section is dedicated to the relation between the Dolbeault integral functor and Donagi-Pantev's classical limit of the geometric Langlands correspondence \cite{DP}. Donagi-Pantev conjectured the existence of a derived equivalence $\DCoh(\Higgs_G(X)) \dashrightarrow \DCoh(\Higgs_{^LG}(X))$ that intertwines a pair of Hecke and Wilson functors, derived from Beilinson-Drinfeld's original conjecture \cite{beilinson&drinfeld} as a classical limit $\hbar \to 0$. In this section, we place the Dolbeault functor $\Phi_{\Dol}$ into a commutative square, thought of as next to Donagi-Pantev's conjecture: 
\begin{equation}
\label{eq approx clGLC}
\begin{tikzcd}[column sep = huge]
\DCoh(X_{\Dol}) 
\arrow[r, "\Phi_{\Dol, \rho}"] \arrow[d, "(.)|_y"'] 
& \DCoh(\Higgs_G(X)) 
\arrow[d, "W_{\mu ,x}"] \arrow[r, dashrightarrow]
& \DCoh(\Higgs_{^LG}(X))
\arrow[d, "^LH_{\mu, x}"]
\\
\DCoh(X_{\Dol}) 
\arrow[r, "\Phi_{\Dol, \rho \otimes \mu}"'] 
& \DCoh(\Higgs_G(X)) 
\arrow[r, dashrightarrow]
& \DCoh(\Higgs_{^LG}(X))
\end{tikzcd}.
\end{equation}
Let us recall the definition of the Hecke and Wilson functors. Let $\Ff$ denote the G-bundle that underlies the universal $G$-Higgs bundle $\Uu = (\Ff, \Theta)$ on $X \times \Higgs_G(X)$. With the cotangent bundle map $\pi: \Higgs_G(X) \to \Bun_G(X)$ defined by the forgetful functor $(P,\theta) \mapsto P$ and the universal $G$-bundle $\Vv$ on $X\times \Higgs_G(X)$, there exists a universal isomorphism 
\begin{equation}
\label{eq universal bundles and pullback}
\Ff \cong (\id \times \pi)^{*}\Vv.
\end{equation}
Then, given a representation $\mu : G \to GL_n$ and a point $x \in X$, the associated Wilson functor is the tensorisation functor 
\[
W_{\mu , x} : \DCoh(\Higgs_G(X)) \to \DCoh(\Higgs_G(X)), \quad \Ee \mapsto \Ee \otimes \mu(\Ff) |_{\{x\} \times \Higgs_G(X)}.
\]
On the other hand, the Hecke functors $^LH_{\mu x}$ are integral functors with a complicated integral kernel $^{^L}\Jj_{\mu, x}$, constructed from analysing the cameral spectral data \cite[§2]{DP}. We shall only need to formulate Hecke transforms on the level of the base curve $X$, as studied in \cite{FGOP} or \cite[§4.2]{hausel&hitchin}, which arise from the standard short exact sequence 
\[
0 \to \Oo_X(-x) \to \Oo_X \to \Oo_X \to 0.
\]
Indeed, given a point $x \in X$ and a sheaf $\Ee$ on $X$, the associated (total) Hecke transform of $\Ee$ at $x$ is 
\[
\Ee \longmapsto \Ee(-x) := \ker(\Ee \to \Ee |_x).
\]
One can repeat the same constructions, 
replacing $X$ with $X_{\Dol}$ and taking a point $y \in X_{\Dol}$. It is then reasonable to consider the restriction functors 
\[
(\cdot)|_{y} : \DCoh(X_{\Dol}) \to \DCoh(X_{\Dol}), \quad \Ee \longmapsto \Ee|_y = \coker(\Ee(-y) \to \Ee) ,
\]
as a "co-Hecke transform". This explains how their appearance in \eqref{eq approx clGLC} can be considered as emulating the classical limit. 

We shall make use of two morphisms between $X$ and $X_{\Dol}$. Recall that $X_{\Dol}$ is the relative classifying stack for the formal group scheme $\widehat{TX} \to X$. There exists a flat map $r_{\Dol} : X_{\Dol} \to X$ induced by the cotangent bundle map $TX \to X$ and a section $\iota_{\Dol} : X \to X_{\Dol}$ induced by the zero section $X \to TX$. The precise interaction between the Dolbeault and Wilson functor can then be stated as follows. 

\begin{theorem}
\label{th exchange wilson hecke}
Fix points $x \in X$ and $y = \iota_{\Dol}(x) \in X_{\Dol}$ related by the inclusion $\iota_{\Dol} : X \to X_{\Dol}$. Let $\rho, \mu : G \to GL_n$ be two representations. The Dolbeault integral functors with respect to $\rho$ and $\rho \otimes \mu$ fit into the following commutative square: 
\begin{equation*}
\begin{tikzcd}[column sep = huge]
\DCoh(X_{\Dol}) \arrow[r, "\Phi_{\Dol, \rho}"] \arrow[d, "(\cdot)|_{y}"'] & \DCoh(\Higgs_G(X)) \arrow[d, "W_{\mu, x}"] \\
\DCoh(X_{\Dol}) \arrow[r, "\Phi_{\Dol, \rho \otimes \mu}"'] & \DCoh(\Higgs_G(X))
\end{tikzcd}.
\end{equation*}
\end{theorem} 
We first require a relation between the universal families $\Vv_{\Dol}$ on $X_{\Dol} \times \Higgs_G(X)$ and $\Vv$ on $X \times \Bun_G(X)$. Let us introduce the notation $(\cdot)_x := (\cdot)|_{\{x\} \times \Higgs_G(X)}$ for the restriction to the "slices" at a point $x \in X$ and similarly for $y \in X_{\Dol}$. 

\begin{lemma} 
\label{le relate V-Dol and F}
Fix points $x \in X$ and $y= \iota_{\Dol}(x) \in X_{\Dol}$. There exists an isomorphism 
\[
\Vv_{\Dol, y} \cong \pi^{*} \Vv_x. 
\]
\end{lemma}

\begin{proof}
The mappings stacks $\Bun_G(X) = \Maps(X,BG)$ and $\Higgs_G(X) = \Maps(X_{\Dol}, BG)$ have evaulation maps 
\[
\ev : X \times \Maps(X,BG) \to BG, \quad \ev_{\Dol} : X_{\Dol} \times \Maps(X_{\Dol}, BG) \to BG.
\]
Let $\gamma$ denote the universal $G$-bundle on BG given by the quotient map $\Spec(\CC) \to BG$. The universal families $\Vv$ and $\Vv_{\Dol}$ associated to the mapping stacks permit the universal isomorphisms 
\[
\Vv \cong \ev^{*}\gamma, \quad \Vv_{\Dol} := \ev_{\Dol}^{*}\gamma.
\]
Pullback along the morphism $r_{\Dol} : X_{\Dol} \to X$ defines the zero section $\sigma_0 = r_{\Dol}^{*} : \Maps(X,BG) \to \Maps(X_{\Dol}, BG)$. Indeed, as a map $\Bun_G(X) \to \Higgs_G(X)$, this is precisely the map $\sigma_0(P) = (P,0)$ that records zero Higgs field. One has a commutative diagram
\begin{equation*}
\begin{tikzcd}[column sep = huge]
BG & X_{\Dol} \times \Maps( X_{\Dol} , BG ) \arrow[l, "\ev_{\Dol}"'] \\
X \times \Maps(X,BG) \arrow[u, "\ev"] \arrow[r, "\id \times \sigma_{0}"'] & X \times \Maps(X_{\Dol}, BG) \arrow[l, bend left, "\id \times \pi"] \arrow[u, "\iota_{\Dol} \times \id"']
\end{tikzcd}.
\end{equation*}
We therefore have the functorial isomorphism
\[
\Vv = \ev^{*}\gamma 
\cong (\id \times \sigma_{0})^{*}(\iota_{\Dol} \times \id)^{*} \ev_{\Dol}^{*} \gamma .
\]
Restricting to $\{x\} \times \Maps(X,BG)$ and taking the pullback along $\pi$ provides
\begin{equation}
\label{eq universal slices}
\pi^{*} \Vv_{x} \cong \Big( \ev_{\Dol} \circ ( \iota_{\Dol}(x) \times (\sigma_{0} \circ \pi )) \Big)^{*} \gamma.
\end{equation}
To analyse the effect of this pullback we evaluate the composition of maps on a point $(x,f) \in X \times \Maps(X_{\Dol},BG)$, 
\begin{align*}
\ev_{\Dol} \circ (\iota_{\Dol} \times (\sigma_{0} \circ \pi)) (x , f) 
& = \ev \Big( \iota_{\Dol}(x) ~ , ~ \sigma_{0} \circ \pi(f) \Big) \\
& = \ev \Big( \iota_{\Dol}(x) ~ , ~ f \circ \iota_{\Dol} \circ r_{\Dol} \Big) \\
& = f \circ \iota_{\Dol} \circ r_{\Dol} \circ \iota_{\Dol} (x) \\
& = f \circ \iota_{\Dol} (x) \\
& = \ev_{\Dol} \circ (\iota_{\Dol} \times \id) ( x, f).
\end{align*}
It follows that the $(\sigma_{0} \circ \pi)^{*}$ term cancels out, so we have an isomorphism
\[
\Big( \ev_{\Dol} \circ ( \iota_{\Dol}(x) \times (\sigma_{0} \circ \pi )) \Big)^{*} \gamma
\cong \Big( \ev_{\Dol} \circ ( \iota_{\Dol}(x) \times \id )\Big)^{*} \gamma.
\]
Substituting this into \eqref{eq universal slices} and applying the universal isomorphism $\Vv_{\Dol} \cong \ev_{\Dol}^{*}\gamma$ yields
\[
\pi^{*}\Vv_x 
\cong (\iota_{\Dol}(x) \times \id)^{*} \ev_{\Dol}^{*} \gamma
= \Vv_{\Dol, y},
\]
which proves the lemma. 
\end{proof}

\begin{proof}(\textit{of Theorem \ref{th exchange wilson hecke}}).
We establish the commutative square with a fixed object $\Gg \in \DCoh(\Higgs_G(X))$. Using Lemma \ref{le relate V-Dol and F} alongside $\Ff \cong (1 \times \pi)^{*}\Vv$ (see \eqref{eq universal bundles and pullback}) we can rewrite the Wilson functor as
\begin{equation}
\label{eq wilson rewritten}
W_{\mu, x}(\Gg) := \Gg \otimes \mu( \Ff_x ) 
\cong \Gg \otimes \mu( \pi^{*}\Vv_x ) 
\cong \Gg \otimes \mu( \Vv_{\Dol, y} ).
\end{equation}
Consider the diagonal morphism $\delta : \Higgs_G(X) \to \Higgs_G(X) \times \Higgs_G(X)$ and the pushforward $\delta_{*}\mu(\Vv_{\Dol, y}) \in \DCoh(\Higgs_G(X) \times \Higgs_G(X))$. The tensorisation functor \eqref{eq wilson rewritten} and the integral functor $\Phi_{\delta_{*}\mu(\Vv_{\Dol, y})}$ with integral kernal $\delta_{*}\mu(\Vv_{\Dol, y})$ satisfy
\[
W_{\mu, x} \simeq \Phi_{\delta_{*}\mu(\Vv_{\Dol, y})}
\]
This allows us to compute $W_{\mu, x} \circ \Phi_{\Dol,\rho}$ as a composition of integral functors, so therefore as a convolution of the integral kernels. Consider the natural projections 
\begin{equation*}
\begin{tikzcd}[column sep = huge]
X_{\Dol} \times \Higgs_G(X) \times \Higgs_G(X) \arrow[d, "\pi_{12}"'] \arrow[dr, "\pi_{13}"'] \arrow[r, "\pi_{23}"] & \Higgs_G(X) \times \Higgs_G(X) \\
X_{\Dol} \times \Higgs_G(X) & X_{\Dol} \times \Higgs_G(X) 
\end{tikzcd}.
\end{equation*}
The convolution of integral kernels is defined to be 
\[
\rho(\Vv_{\Dol}) \star \delta_{*}\mu(\Vv_{\Dol, y}) := \pi_{13,*}\Big(\pi_{12}^{*}\rho(\Vv_{\Dol}) \otimes \pi_{23}^{*} \circ \delta_{*}\mu(\Vv_{\Dol, y}) \Big). 
\]
It follows that transforming an object $\Ee \in \DCoh(X_{\Dol})$ results in equivalences 
\begin{align*}
W_{\mu, x} \circ \Phi_{\Dol, \rho} (\Ee) 
& \cong \Phi_{\delta_{*}\mu(\Vv_{\Dol, y})} \circ \Phi_{\Dol, \rho} (\Ee) \\
& \cong \Phi_{\rho(\Vv_{\Dol}) \star \delta_{*}\mu(\Vv_{\Dol, y})}(\Ee) \\
& = Rp_{2,*} \Big( \pi_{13,*} \Big( \pi_{12}^{*}\rho(\Vv_{\Dol}) \otimes \pi_{23}^{*} \circ \delta_{*}\mu(\Vv_{\Dol, y}) \Big)  \otimes p_1^{*}\Ee \Big).
\end{align*}
With respect to the Cartesian diagram 
\begin{equation*}
\begin{tikzcd} 
X_{\Dol} \times \Higgs_G(X)
\arrow[r, "\id \times \delta"] 
\arrow[d, "p_2"'] 
\arrow[dr, phantom, "\square"] 
& X_{\Dol} \times \Higgs_G(X) \times \Higgs_G(X) 
\arrow[d, "\pi_{23}"] 
\\
\Higgs_G(X) \arrow[r, "\delta"'] 
& \Higgs_G(X) \times \Higgs_G(X)
\end{tikzcd},
\end{equation*} 
base change provides $\pi_{23}^{*} \circ \delta_{*}  \simeq (1 \times \delta)_{*} \circ p_2^{*}$. Our calculations continue with 
\[
W_{\Dol, \mu, x} \circ \Phi_{\Dol, \rho} (\Ee) 
\cong Rp_{2,*} 
\Big( 
\pi_{13,*} \Big( \pi_{12}^{*}\rho(\Vv_{\Dol}) \otimes (1 \times \delta)_{*} \circ p_2^{*}\mu(\Vv_{\Dol ,y}) \Big)  \otimes p_1^{*}\Ee
\Big). 
\]
The final observation is that we are applying pushforward along $\pi_{13}$ on a sheaf that is supported on the diagonal in $X_{\Dol} \times \Higgs_G(X) \times \Higgs_G(X)$. The action of $\pi_{13, *}$ is therefore equivalent to $\pi_{12, *}$ on this sheaf. Alongside the projection formula and $\pi_{12} \circ (\id \times \delta) = \id$, this allows us to calculate 
\begin{align*}
W_{\Dol, \mu, x} \circ \Phi_{\Dol, \rho} (\Ee) 
& \cong Rp_{2,*} 
\Big( 
\rho(\Vv_{\Dol}) \otimes \pi_{12,*} \circ (1 \times \delta)_{*} \circ p_2^{*}\mu(\Vv_{\Dol, y})  \otimes p_1^{*}\Ee 
\Big) 
\\
& \cong Rp_{2,*} 
\Big( 
\rho(\Vv_{\Dol}) \otimes p_2^{*}\mu(\Vv_{\Dol, y})  \otimes p_1^{*}\Ee
\Big) 
\\
& \cong 
Rp_{2,*} 
\Big( 
\rho(\Vv_{\Dol}) \otimes \mu(\Vv_{\Dol}) \otimes \Oo |_{\{y\} \times \Higgs_G(X)} \otimes p_1^{*}\Ee 
\Big) 
\\
& \cong 
Rp_{2,*} 
\Big( 
\rho(\Vv_{\Dol}) \otimes \mu(\Vv_{\Dol})  \otimes p_1^{*} (\Oo|_{y} \otimes \Ee )
\Big) 
\\
& = 
Rp_{2,*} 
\Big( 
\rho \otimes \mu (\Vv_{\Dol})  \otimes p_1^{*} ( \Ee|_{y} )
\Big),
\end{align*}
which concludes the proof. 
\end{proof}

\end{document}